\documentclass[11pt]{article}
\usepackage[left=1in,top=1in,right=1in]{geometry}
\pdfoutput=1 
\usepackage[utf8]{inputenc}
\usepackage[T1]{fontenc}
\usepackage[english]{babel}
\usepackage{amssymb,amsmath,amsthm,bm,cmap,mathtools,enumitem}
\usepackage{tikz}
\usetikzlibrary{arrows}
\usepackage{graphicx}
\usepackage{hyperref}
\usepackage[nameinlink]{cleveref}
\usepackage{listings} 

\newtheorem{theorem}{Theorem}
\newtheorem{lemma}[theorem]{Lemma}

\newtheorem{observation}[theorem]{Observation}
\newtheorem{proposition}[theorem]{Proposition}
\newtheorem{conjecture}[theorem]{Conjecture}

\newtheorem{problem}[theorem]{Problem}

\newcommand*{\eqdef}{\stackrel{\mbox{\normalfont\tiny def}}{=}} 
\newcommand*{\veps}{\varepsilon}                                

\DeclareMathOperator{\bB}{\bf{B}}
\DeclareMathOperator{\bD}{\bf{D}}
\DeclareMathOperator{\bT}{\bf{T}}
\DeclareMathOperator{\bO}{\bf{O}}
\DeclareMathOperator{\oul}{{\bf{O}}_{UL}}
\DeclareMathOperator{\our}{{\bf{O}}_{UR}}
\DeclareMathOperator{\odl}{{\bf{O}}_{DL}}
\DeclareMathOperator{\odr}{{\bf{O}}_{DR}}
\DeclareMathOperator{\oll}{{\bf{O}}_L}
\DeclareMathOperator{\orr}{{\bf{O}}_R}
\DeclareMathOperator{\conv}{conv}
\DeclareMathOperator{\slp}{\sf{slope}}
\DeclareMathOperator{\lin}{\sf{line}}
\DeclareMathOperator{\ray}{\sf{ray}}
\DeclareMathOperator{\ram}{ram}
\DeclareMathOperator{\sat}{sat}

\title{Saturation results around the Erdős--Szekeres problem}
\author{Gábor Damásdi\thanks{ELTE Eötvös Loránd University and Alfr\'ed R\'enyi Institute of Mathematics, Budapest, Hungary.  Research supported by ERC grant No.~882971, ``GeoScape''. {\tt damasdigabor@caesar.elte.hu}.}
\and Zichao Dong\thanks{Alfr\'ed R\'enyi Institute of Mathematics, Budapest, Hungary and Extremal Combinatorics and Probability Group (ECOPRO), Institute for Basic Science (IBS), Daejeon, South Korea. Research supported by ERC grant No.~882971, ``GeoScape'', by the Erd\H os Center, and by the Institute for Basic Science (IBS-R029-C4). {\tt zichao@ibs.re.kr}.}
\and Manfred Scheucher\thanks{Technische Universität Berlin,
Institut für Mathematik, Germany. Supported by DFG grant SCHE 2214/1-1. {\tt scheucher@math.tu-berlin.de}.}
\and Ji Zeng\thanks{University of California San Diego, La Jolla, CA and Alfr\'ed R\'enyi Institute of Mathematics, Budapest, Hungary. Research supported by NSF grant DMS-1800746, ERC grant No.~882971, ``GeoScape'', and by the Erd\H os Center. {\tt jzeng@ucsd.edu}.}}
\date{}

\begin{document}

\maketitle

\begin{abstract}
    In this paper, we consider saturation problems related to the celebrated Erd\H{o}s--Szekeres convex polygon problem. For each $n \ge 7$, we construct a planar point set of size $(7/8) \cdot 2^{n-2}$ which is saturated for convex $n$-gons. That is, the set contains no $n$ points in convex position while the addition of any new point creates such a configuration. This demonstrates that the saturation number is smaller than the Ramsey number for the Erd\H{o}s--Szekeres problem. The proof also shows that the original Erd\H{o}s--Szekeres construction is indeed saturated.

    Our construction is based on a similar improvement for the saturation version of the cups-versus-caps theorem. Moreover, we consider the generalization of the cups-versus-caps theorem to monotone paths in ordered hypergraphs. In contrast to the geometric setting, we show that this abstract saturation number is always equal to the corresponding Ramsey number.
\end{abstract}

\section{Introduction}\label{sec:intro}

Two types of problems are of central significance in extremal combinatorics. Tur\'an-type problems originate from the work of Tur\'an \cite{turan41} (earlier Mantel~\cite{mantel}) that determines the maximum number of edges in a graph without a $k$-clique. Ramsey-type problems begin with~\cite{ramsey1930problem} which states that any large enough graph must contain either a $k$-clique or a $k$-independent set.

In 1964, Erdős, Hajnal, and Moon~\cite{erdos1964problem} investigated a variation of Tur\'an's theorem, called \emph{the saturation problem}, where one aims to minimize the number of edges in a graph that does not contain a $k$-clique, but the addition of any edge to this graph yields a $k$-clique. More generally, saturation problems can be considered in various settings, see our incomplete list \cite{furedi1980maximal,kaszonyi1986saturated,furedi2013cycle,ferrara2017saturation,DKMTWZ2021}. Typically, an object that is maximal with respect to certain property is said to be \textit{saturated}. While the classical extremal problems such as Tur\'an-type problems or Ramsey-type problems ask for certain maximum quantity possibly achieved, their saturation versions aim at the corresponding minimum quantity possibly achieved by saturated objects.

The Erdős--Szekeres problem is a classical extremal problem in combinatorial geometry proposed in the seminal paper \cite{ErdosSzekeres1935} back to 1935. It asks for the maximum size of a planar point set that does not contain $k$ points in convex position. In this paper, we consider the saturation version of the Erdős--Szekeres problem, as well as the saturation versions of the related Ramsey-type results in \cite{ErdosSzekeres1935} and some later graph-theoretic generalizations.

\subsection{The Erdős--Szekeres lemma on monotone sequences}

Let $\ram_{\sf s}(k,\ell)$ be the maximum length of a sequence of distinct real numbers that is \textit{$(k, \ell)$-seq.-free} (i.e.~containing no increasing subsequence of length $k$ or decreasing subsequence of length $\ell$). The first result in \cite{ErdosSzekeres1935}, later known as the Erdős-Szekeres lemma, states that
\begin{equation}\label{eq:ESL}
    \ram_{\sf s}(k,\ell) = (k-1)(\ell-1).
\end{equation}
Like most problems considered in this paper, this extremal quantity can be considered as Ramsey-type where the edge coloring is whether each pair is increasing or decreasing.

In \cite{DKMTWZ2021}, the authors studied the saturation version of \eqref{eq:ESL}. In their setting, a sequence is called \textit{$(k,\ell)$-seq.-saturated} if it is $(k, \ell)$-seq.-free and the insertion of any distinct real number anywhere into this sequence yields such a monotone subsequence. The saturation number $\sat_{\sf s}(k,\ell)$ is the minimum length of a sequence that is $(k,\ell)$-seq.-saturated. It is obvious that $\sat_{\sf s}(k,\ell)\leq \ram_{\sf s}(k,\ell)$. Interestingly, the saturation number and the Ramsey number are equal in this setting. 

\begin{theorem}[Damásdi et al.~\cite{DKMTWZ2021}]\label{thm:sat_s}
    For any integers $k, \ell \ge 1$, we have $\sat_{\sf s}(k,\ell)=\ram_{\sf s}(k,\ell)$. 
\end{theorem}

Later in this paper, we shall give a new proof of this theorem. Our proof leads to a hypergraph variation of this result, see \Cref{thm:sat_p}. 

\subsection{The Erdős--Szekeres theorem on cups-versus-caps}

The main objects we consider in this paper are planar point sets. To simplify our discussion, a planar point set is said to be \textit{generic} if it is \textit{in general position} (meaning without three collinear points throughout this paper) and its members have distinct $x$-coordinates. A point $p$ is said to be \textit{generic} with respect to a set $P$ if $P\cup\{p\}$ is a generic point set.

For a sequence of generic points $p_1,\dots,p_n$ ordered increasingly with respect to $x$-coordinates, we say that $p_1, \dots, p_n$ form a \textit{cup} (resp.~\textit{cap}) if the slopes of the lines $p_ip_{i+1}$ (for $1 \le i < n$) form an increasing (resp.~decreasing) sequence. Moreover, a \textit{$k$-cup} refers to a cup of size $k$ and an \textit{$\ell$-cap} refers to a cap of size $\ell$. Let $\ram_{\sf c}(k,\ell)$ denote the maximum size of a generic planar point set that is \textit{$(k, \ell)$-cup-cap-free} (i.e.~containing neither a $k$-cup nor an $\ell$-cap). The second result in \cite{ErdosSzekeres1935}, later known as the Erdős--Szekeres theorem, states that
\begin{equation}\label{eq:EST}
    \ram_{\sf c}(k,\ell) = \binom{k+\ell-4}{k-2}. 
\end{equation}

In this paper, we study the saturation version of \eqref{eq:EST}. A generic planar point set is said to be \textit{$(k,\ell)$-cup-cap-saturated} if it is $(k, \ell)$-cup-cap-free and the addition of any generic point with respect to the set yields a $k$-cup or an $\ell$-cap. The saturation number $\sat_{\sf c}(k,\ell)$ is defined as the minimum size of a point set that is $(k,\ell)$-cup-cap-saturated. In contrast to \Cref{thm:sat_s}, we will show that $\sat_{\sf c}(k,\ell)$ is in general significantly smaller than $\ram_{\sf c}(k,\ell)$. 

\begin{theorem}\label{thm:sat_c}
    For all integers $k,\ell \ge 4$, we have 
    \begin{equation*}
        2k + 2\ell -14 \leq \sat_{\sf c}(k, \ell) \leq \binom{k+\ell-4}{k-2} - 2\binom{k+\ell-8}{k-4}. 
    \end{equation*}
\end{theorem}

We write $\binom{-1}{0} = 0$ and $\binom{0}{0} = 1$. In Section~\ref{sec:cupcap} where this theorem is proven, we also determine $\sat_{\sf c}(k, \ell)$ for small values of $(k, \ell)$. For example, we establish that $\sat_{c}(4,5) = 8 < 10 = \ram_{\sf c}(4,5)$, which is the lexicographically smallest case where ``$\sat_{\sf c}(k,\ell) < \ram_{\sf c}(k,\ell)$'' happens. 

\subsection{The Erdős--Szekeres problem on convex polygons}

Another landmark in the paper \cite{ErdosSzekeres1935} of Erdős and Szekeres is the following famous conjecture. 

\begin{conjecture}
    Every set of $2^{n-2}+1$ points in the plane that are in general position contains $n$ points in convex position, and this bound is tight in the worst case. 
\end{conjecture}

Denote by $\ram_{\sf g}(n)$ the maximum size of a planar point set in general position that is \textit{$n$-gon-free} (i.e.~containing no $n$ points in convex position). The Erdős--Szekeres conjecture can be phrased as
\begin{equation*}
    \text{Is it true that $\ram_{\sf g}(n)= 2^{n-2}$ for all $n \geq 2$?}
\end{equation*}
In a subsequent paper \cite{ErdosSzekeres1961} from 1961, Erdős and Szekeres constructed a generic point set of size $2^{n-2}$ that is $n$-gon-free for all integers $n \geq 2$. Thus, for any integer $n \ge 2$, we have
\begin{equation}\label{eq:ES_construction}
    \ram_{\sf g}(n) \geq 2^{n-2}. 
\end{equation}
Notice that $\ram_{\sf g}(n) \leq \ram_{\sf c}(n,n)$, since any $n$-cup or $n$-cap is an $n$-gon. Hence, \eqref{eq:EST} implies that $\ram_{\sf g}(n)\leq \binom{2n-4}{n-2} = 4^{n-o(n)}$. There have been only small improvements of this upper bound, until Suk \cite{Suk2017} made a breakthrough in 2017 by proving that $\ram_{\sf g}(n) \leq 2^{n+o(n)}$ (see also \cite{HMPT2017}). The Erdős--Szekeres conjecture has been verified for $n \leq 6$ in \cite{SzekeresPeters2006}.

We consider the saturation version of the Erd\H{o}s--Szekeres problem. A planar point set $P$ is $n$-\emph{gon}-\emph{saturated} if it is $n$-gon-free while any $q \notin P$ with $P \cup \{q\}$ being in general position is part of an $n$-gon in $P \cup \{q\}$. Denote by $\sat_{\sf g}(n)$ the smallest size of an $n$-gon-saturated set. The main result of our paper is that $\sat_{\sf g}(n)$ is significantly smaller than $\ram_{\sf g}(n)$ in general. 

\begin{theorem}\label{thm:sat_g}
    For any integer $n\geq 7$, we have $\sat_{\sf g}(n) \leq \frac{7}{8} \cdot 2^{n-2} \leq \frac{7}{8} \cdot \ram_{\sf g}(n)$.
\end{theorem}

We prove this theorem by modifying the original construction of \cite{ErdosSzekeres1961}, replacing the sub-structures in their point set with smaller ones we found in Theorem~\ref{thm:sat_c}. The proof also shows that the original Erd\H{o}s--Szekeres construction in \cite{ErdosSzekeres1961} is indeed saturated. As far as we know, this is widely believed (for otherwise the Erdős--Szekeres conjecture would be disproved immediately) but never been verified in the literature since its existence from 1961.

\subsection{Graph-theoretic generalizations}

The increasing or decreasing sequences and cups or caps can be generalized to monotone paths in a graph-theoretic setting. Inside any $r$-uniform complete hypergraph $H$ with a linear-ordered vertex set $V(H)$, a \textit{monotone path of length} $k$ (or a \textit{monotone $k$-path}) is a subgraph of $H$ with vertices being $v_1<v_2<\dots<v_{k+r-1}$ and (hyper)edges being $\{v_i,v_{i+1},\dots,v_{i+r-1}\}$ for $i = 1, \dots, k$. Here we define the length of a monotone path as the number of its edges rather than its vertices.

Let $\ram^{(r)}_{\sf p}(k,\ell)$ be the maximum size of an $r$-uniform vertex-ordered complete hypergraph $H$ that is \textit{$(k, \ell)$-path-free} (i.e.~each edge is colored by one of red and blue such that $H$ contains neither a red monotone $k$-path nor a blue monotone $\ell$-path). The determination of $\ram^{(r)}_{\sf p}(k,\ell)$ is a Ramsey-type problem that connects to the Erd\H{o}s--Szekeres results in the following way: Given a generic planar point set $P$, we can create a $3$-uniform complete hypergraph $H_P$ whose vertices are $P$ ordered by their $x$-coordinates. Color each triple of $H_P$ red or blue based on whether it is a cup or a cap. Then a monochromatic monotone $\ell$-path in $H_P$ corresponds to an $(\ell+2)$-cup or cap. So, $\ram_{\sf p}^{(3)}(k,\ell) \geq \ram_{\sf c}(k+2,\ell+2)$. Similarly, $\ram_{\sf p}^{(2)}(k,\ell) \geq \ram_{\sf s}(k+1,\ell+1)$. The Ramsey numbers $\ram^{(r)}_{\sf p}(k,\ell)$ are extensively studied, and their bounds are considered as abstract generalizations of \eqref{eq:ESL} and \eqref{eq:EST}. See \cite{eliavs2013higher,fox2012erdHos,MoshkovitzShapira2014} for more details. 

Another major question investigated in this paper is the saturation problem for monotone paths. We say that a $2$-colored vertex-ordered complete hypergraph $H$ is \textit{$(k,\ell)$-path-saturated} if $H$ contains neither red monotone $k$-path nor blue monotone $\ell$-path, and any $H^+$ properly containing $H$ contains either a red monotone $k$-path or a blue monotone $\ell$-path. Here, $H^+$ is again a $2$-colored vertex-ordered complete hypergraph such that the containment preserves both the ordering and the coloring. The saturation number $\sat_{\sf p}^{(r)}(k,\ell)$ is defined to be the minimum size of an $r$-uniform $2$-colored ordered complete hypergraph that is $(k,\ell)$-path-saturated. In contrast to the geometric setting of \Cref{thm:sat_c,thm:sat_g}, we show that this abstract saturation number is always equal to the Ramsey number. 
\begin{theorem}\label{thm:sat_p}
    For any integers $r \geq 2$ and $k,\ell \geq 1$, we have $\sat^{(r)}_{\sf p}(k,\ell) = \ram^{(r)}_{\sf p}(k,\ell)$.
\end{theorem}

\smallskip

The rest of this paper is organized as follows: In \Cref{sec:warmup}, we present a new proof of \Cref{thm:sat_s} which serves as an illustration of our proof of \Cref{thm:sat_p}. In \Cref{sec:hypergraph}, we use a labeling technique of Moshkovitz and Shapira \cite{MoshkovitzShapira2014} to prove \Cref{thm:sat_p} completely. \Cref{sec:cupcap} is devoted to the saturation problem for cups-versus-caps. In particular, we prove \Cref{thm:sat_c} there. \Cref{sec:gon} is devoted to the proof of \Cref{thm:sat_g}. Finally, we include remarks and open problems in \Cref{sec:remarks}.

\section{Warm-up: monotone paths in ordered graphs}\label{sec:warmup}

The following proof of the Erdős--Szekeres Lemma is due to Seidenberg \cite{Seidenberg1959}. For a $(k,\ell)$-seq.-free  sequence $v_1,\dots,v_n$ assign a vertex label $L(v_i) = (k_i,\ell_i)$, where 
\begin{itemize}
    \item $k_i$ is the length of the longest increasing subsequence ending at $v_i$, and 
    \item $\ell_i$ is the length of the longest decreasing subsequence ending at $v_i$. 
\end{itemize}
For $i<j$, either the decreasing or the increasing sequence ending at $v_i$ can be extended to $v_j$. So, $L(v_i)\ne L(v_j)$ for any $i\ne j$. Therefore, there are at most $(k-1)(\ell-1)$ elements in the sequence. 

The main idea behind the proof of \Cref{thm:sat_p} is to analyze a similar vertex labeling procedure, which was introduced in the general case by Moshkovitz and Shapira \cite{MoshkovitzShapira2014}. If we have less than $\ram_{\sf p}^{(r)}(k,\ell)$ points, then after creating the vertex labeling, one of the possible labels is missing. We will extend the hypergraph and its coloring so that the new vertex receives one of the missing labels while the labels of other vertices do not change. As a warm-up, we prove $\sat_{\sf p}^{(2)}(k,\ell) = \ram_{\sf p}^{(2)}(k,\ell)$. 

\begin{proof}[Proof of \Cref{thm:sat_p} assuming $r = 2$]
      Consider a red-blue colored complete graph $H$ on the vertex set $V(H) = [N]\eqdef \{1,2,\dots,N\}$ under the usual ordering ``$<$'' so that there is no red path of length $k$ or blue path of length $\ell$. Assume that $N<\ram_{\sf p}^{(2)}(k, \ell)$. We need to show that we can add a new vertex and extend the coloring to the new edges.
    
    For each vertex $v$, denote by $L(v)$ the pair $(1 + \ell_{\sf r}, 1 + \ell_{\sf b})$, where 
    \begin{itemize}
        \item $\ell_{\sf r}$ is the length of the longest red monotone path ending at $v$, and
        \item $\ell_{\sf b}$ is the length of the longest blue monotone path ending at $v$. 
    \end{itemize}
    Recall that the length of a path is the number of its edges, that is, for $r=2$ the length is one less than the number of vertices in the path. Call $L(v)$ the \emph{vertex label} of $v$, and define a partial order ``$\preceq$'' on pairs such that $(x_1,x_2) \preceq (y_1,y_2)$ if and only if $x_1 \le y_1$ and $x_2 \le y_2$. Then, a coloring is $(k, \ell)$-path-free if and only if $L(v) \preceq (k, \ell)$ for all $v\in V(H)$. 
    
    If $v_1 < v_2$ and $v_1v_2$ is red, then $v_1v_2$ extends any red path ending at $v_1$, and if $v_1v_2$ is blue it extends the blue paths. So, $L(v_2)$ is strictly larger than $L(v_1)$ in at least one coordinate. 

    \begin{observation}\label{obs:ord}
    If $v_1 < v_2$, then $L(v_1) \not\succeq L(v_2)$. In particular, $v_1 \neq v_2$ implies $L(v_1) \neq L(v_2)$. 


    \end{observation}
 
    \subsubsection*{Picking the position}

    Since $\ram_{\sf p}^{(2)}(k,\ell) = (k-1)(\ell-1)$, there must be a pair $(\ell_{\sf r},\ell_{\sf b}) \preceq (k,\ell)$ which does not appear as a vertex label. As an example, in \Cref{fig:mon_paths_example} we have $k=2, \, \ell=3$ and the vertex label $(1,3)$ is missing. Let $(\ell_{\sf r},\ell_{\sf b})$ be a missing vertex label which is minimal with respect to $\preceq$. That is, any pair $(x,y) \preceq (\ell_{\sf r}, \ell_{\sf b})$ appears but $(\ell_{\sf r}, \ell_{\sf b})$ does not. 

    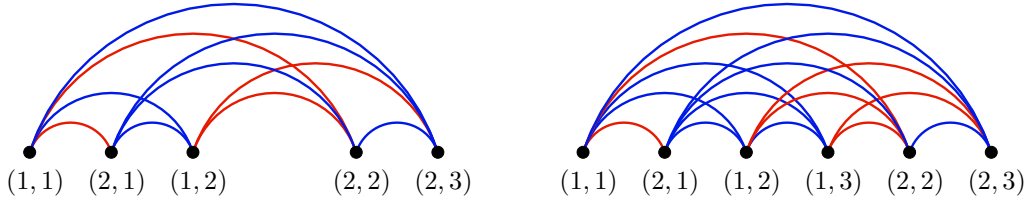
\begin{figure}[!ht]
        \centering
        \definecolor{qqqqff}{rgb}{0.,0.1,0.9}
    \definecolor{ffqqqq}{rgb}{0.9,0.1,0.}
    \scalebox{0.9}{
    \begin{tikzpicture}[line cap=round,line join=round,>=triangle 45,x=0.6cm,y=0.6cm]
    \clip(3.4348337503791915,-1.1) rectangle (15,3.932075376066969);
    \draw [shift={(5.,-0.3249196962329061)},line width=1.pt,color=ffqqqq]  plot[domain=0.31415926535897915:2.8274333882308142,variable=\t]({1.*1.0514622242382672*cos(\t r)+0.*1.0514622242382672*sin(\t r)},{0.*1.0514622242382672*cos(\t r)+1.*1.0514622242382672*sin(\t r)});
    \draw [shift={(6.,-0.6498393924658125)},line width=1.pt,color=qqqqff]  plot[domain=0.31415926535897926:2.827433388230814,variable=\t]({1.*2.1029244484765344*cos(\t r)+0.*2.1029244484765344*sin(\t r)},{0.*2.1029244484765344*cos(\t r)+1.*2.1029244484765344*sin(\t r)});
    \draw [shift={(7.,-0.32491969623290584)},line width=1.pt,color=qqqqff]  plot[domain=0.31415926535897887:2.8274333882308142,variable=\t]({1.*1.051462224238267*cos(\t r)+0.*1.051462224238267*sin(\t r)},{0.*1.051462224238267*cos(\t r)+1.*1.051462224238267*sin(\t r)});
    \draw [shift={(10.,-0.6498393924658122)},line width=1.pt,color=ffqqqq]  plot[domain=0.31415926535897915:2.8274333882308142,variable=\t]({1.*2.1029244484765344*cos(\t r)+0.*2.1029244484765344*sin(\t r)},{0.*2.1029244484765344*cos(\t r)+1.*2.1029244484765344*sin(\t r)});
    \draw [shift={(11.,-0.9747590886987186)},line width=1.pt,color=ffqqqq]  plot[domain=0.31415926535897926:2.827433388230814,variable=\t]({1.*3.1543866727148018*cos(\t r)+0.*3.1543866727148018*sin(\t r)},{0.*3.1543866727148018*cos(\t r)+1.*3.1543866727148018*sin(\t r)});
    \draw [shift={(13.,-0.32491969623290634)},line width=1.pt,color=qqqqff]  plot[domain=0.3141592653589793:2.827433388230814,variable=\t]({1.*1.0514622242382672*cos(\t r)+0.*1.0514622242382672*sin(\t r)},{0.*1.0514622242382672*cos(\t r)+1.*1.0514622242382672*sin(\t r)});
    \draw [shift={(8.,-1.2996787849316247)},line width=1.pt,color=ffqqqq]  plot[domain=0.3141592653589792:2.827433388230814,variable=\t]({1.*4.205848896953069*cos(\t r)+0.*4.205848896953069*sin(\t r)},{0.*4.205848896953069*cos(\t r)+1.*4.205848896953069*sin(\t r)});
    \draw [shift={(9.,-0.9747590886987186)},line width=1.pt,color=qqqqff]  plot[domain=0.31415926535897926:2.827433388230814,variable=\t]({1.*3.1543866727148018*cos(\t r)+0.*3.1543866727148018*sin(\t r)},{0.*3.1543866727148018*cos(\t r)+1.*3.1543866727148018*sin(\t r)});
    \draw [shift={(9.,-1.6245984811645304)},line width=1.pt,color=qqqqff]  plot[domain=0.3141592653589791:2.8274333882308142,variable=\t]({1.*5.257311121191336*cos(\t r)+0.*5.257311121191336*sin(\t r)},{0.*5.257311121191336*cos(\t r)+1.*5.257311121191336*sin(\t r)});
    \draw [shift={(10.,-1.2996787849316245)},line width=1.pt,color=qqqqff]  plot[domain=0.31415926535897915:2.8274333882308142,variable=\t]({1.*4.205848896953069*cos(\t r)+0.*4.205848896953069*sin(\t r)},{0.*4.205848896953069*cos(\t r)+1.*4.205848896953069*sin(\t r)});
    \draw (3.2,-0.2) node[anchor=north west] {$(1,1)$};
    \draw (13.2,-0.2) node[anchor=north west] {$(2,3)$};
    \draw (11.2,-0.2) node[anchor=north west] {$(2,2)$};
    \draw (7.2,-0.2) node[anchor=north west] {$(1,2)$};
    \draw (5.2,-0.2) node[anchor=north west] {$(2,1)$};
    \begin{scriptsize}
    \draw [fill=black] (4.,0.) circle (2.5pt);
    \draw [fill=black] (6.,0.) circle (2.5pt);
    \draw [fill=black] (8.,0.) circle (2.5pt);
    \draw [fill=black] (12.,0.) circle (2.5pt);
    \draw [fill=black] (14.,0.) circle (2.5pt);
    \end{scriptsize}
    \end{tikzpicture}}
    \hspace{20px}
    \scalebox{0.9}{
    \begin{tikzpicture}[line cap=round,line join=round,>=triangle 45,x=0.6cm,y=0.6cm]
    \clip(3.4348337503791915,-1.1) rectangle (15,3.932075376066969);
    \draw [shift={(5.,-0.3249196962329061)},line width=1.pt,color=ffqqqq]  plot[domain=0.31415926535897915:2.8274333882308142,variable=\t]({1.*1.0514622242382672*cos(\t r)+0.*1.0514622242382672*sin(\t r)},{0.*1.0514622242382672*cos(\t r)+1.*1.0514622242382672*sin(\t r)});
    \draw [shift={(6.,-0.6498393924658125)},line width=1.pt,color=qqqqff]  plot[domain=0.31415926535897926:2.827433388230814,variable=\t]({1.*2.1029244484765344*cos(\t r)+0.*2.1029244484765344*sin(\t r)},{0.*2.1029244484765344*cos(\t r)+1.*2.1029244484765344*sin(\t r)});
    \draw [shift={(7.,-0.32491969623290584)},line width=1.pt,color=qqqqff]  plot[domain=0.31415926535897887:2.8274333882308142,variable=\t]({1.*1.051462224238267*cos(\t r)+0.*1.051462224238267*sin(\t r)},{0.*1.051462224238267*cos(\t r)+1.*1.051462224238267*sin(\t r)});
    \draw [shift={(7.,-0.9747590886987185)},line width=1.pt,color=qqqqff]  plot[domain=0.3141592653589792:2.827433388230814,variable=\t]({1.*3.1543866727148018*cos(\t r)+0.*3.1543866727148018*sin(\t r)},{0.*3.1543866727148018*cos(\t r)+1.*3.1543866727148018*sin(\t r)});
    \draw [shift={(8.,-0.6498393924658122)},line width=1.pt,color=qqqqff]  plot[domain=0.31415926535897915:2.8274333882308142,variable=\t]({1.*2.1029244484765344*cos(\t r)+0.*2.1029244484765344*sin(\t r)},{0.*2.1029244484765344*cos(\t r)+1.*2.1029244484765344*sin(\t r)});
    \draw [shift={(9.,-0.32491969623290634)},line width=1.pt,color=qqqqff]  plot[domain=0.3141592653589793:2.827433388230814,variable=\t]({1.*1.0514622242382672*cos(\t r)+0.*1.0514622242382672*sin(\t r)},{0.*1.0514622242382672*cos(\t r)+1.*1.0514622242382672*sin(\t r)});
    \draw [shift={(10.,-0.6498393924658122)},line width=1.pt,color=ffqqqq]  plot[domain=0.31415926535897915:2.8274333882308142,variable=\t]({1.*2.1029244484765344*cos(\t r)+0.*2.1029244484765344*sin(\t r)},{0.*2.1029244484765344*cos(\t r)+1.*2.1029244484765344*sin(\t r)});
    \draw [shift={(11.,-0.32491969623290634)},line width=1.pt,color=ffqqqq]  plot[domain=0.3141592653589793:2.827433388230814,variable=\t]({1.*1.0514622242382672*cos(\t r)+0.*1.0514622242382672*sin(\t r)},{0.*1.0514622242382672*cos(\t r)+1.*1.0514622242382672*sin(\t r)});
    \draw [shift={(11.,-0.9747590886987186)},line width=1.pt,color=ffqqqq]  plot[domain=0.31415926535897926:2.827433388230814,variable=\t]({1.*3.1543866727148018*cos(\t r)+0.*3.1543866727148018*sin(\t r)},{0.*3.1543866727148018*cos(\t r)+1.*3.1543866727148018*sin(\t r)});
    \draw [shift={(12.,-0.6498393924658127)},line width=1.pt,color=ffqqqq]  plot[domain=0.3141592653589793:2.827433388230814,variable=\t]({1.*2.1029244484765344*cos(\t r)+0.*2.1029244484765344*sin(\t r)},{0.*2.1029244484765344*cos(\t r)+1.*2.1029244484765344*sin(\t r)});
    \draw [shift={(13.,-0.32491969623290634)},line width=1.pt,color=qqqqff]  plot[domain=0.3141592653589793:2.827433388230814,variable=\t]({1.*1.0514622242382672*cos(\t r)+0.*1.0514622242382672*sin(\t r)},{0.*1.0514622242382672*cos(\t r)+1.*1.0514622242382672*sin(\t r)});
    \draw [shift={(8.,-1.2996787849316247)},line width=1.pt,color=ffqqqq]  plot[domain=0.3141592653589792:2.827433388230814,variable=\t]({1.*4.205848896953069*cos(\t r)+0.*4.205848896953069*sin(\t r)},{0.*4.205848896953069*cos(\t r)+1.*4.205848896953069*sin(\t r)});
    \draw [shift={(9.,-0.9747590886987186)},line width=1.pt,color=qqqqff]  plot[domain=0.31415926535897926:2.827433388230814,variable=\t]({1.*3.1543866727148018*cos(\t r)+0.*3.1543866727148018*sin(\t r)},{0.*3.1543866727148018*cos(\t r)+1.*3.1543866727148018*sin(\t r)});
    \draw [shift={(9.,-1.6245984811645304)},line width=1.pt,color=qqqqff]  plot[domain=0.3141592653589791:2.8274333882308142,variable=\t]({1.*5.257311121191336*cos(\t r)+0.*5.257311121191336*sin(\t r)},{0.*5.257311121191336*cos(\t r)+1.*5.257311121191336*sin(\t r)});
    \draw [shift={(10.,-1.2996787849316245)},line width=1.pt,color=qqqqff]  plot[domain=0.31415926535897915:2.8274333882308142,variable=\t]({1.*4.205848896953069*cos(\t r)+0.*4.205848896953069*sin(\t r)},{0.*4.205848896953069*cos(\t r)+1.*4.205848896953069*sin(\t r)});
    \draw (3.2,-0.2) node[anchor=north west] {$(1,1)$};
    \draw (13.2,-0.2) node[anchor=north west] {$(2,3)$};
    \draw (11.2,-0.2) node[anchor=north west] {$(2,2)$};
    \draw (9.2,-0.2) node[anchor=north west] {$(1,3)$};
    \draw (7.2,-0.2) node[anchor=north west] {$(1,2)$};
    \draw (5.2,-0.2) node[anchor=north west] {$(2,1)$};
    \begin{scriptsize}
    \draw [fill=black] (4.,0.) circle (2.5pt);
    \draw [fill=black] (6.,0.) circle (2.5pt);
    \draw [fill=black] (8.,0.) circle (2.5pt);
    \draw [fill=black] (10.,0.) circle (2.5pt);
    \draw [fill=black] (12.,0.) circle (2.5pt);
    \draw [fill=black] (14.,0.) circle (2.5pt);
    \end{scriptsize}
    \end{tikzpicture}}
        \caption{Extending $H$ and its coloring over a new vertex with the label (1,3).}
        \label{fig:mon_paths_example}
    \end{figure}

    Our goal is to add a new vertex $v^+$ into $V(H)$ (specify its position with respect to $\preceq$ among the vertex labels of the original $N$ vertices), and extend the coloring to edges emanated from $v^+$ such that after recomputing the vertex labels we get $L(v^+) = (\ell_{\sf r},\ell_{\sf b})$ and the vertex label of each original vertex is unchanged. Let $w$ be the last vertex such that $L(w) \prec (\ell_{\sf r}, \ell_{\sf b})$. Here $w$ exists because $L(1) = (1, 1)$. Introduce the new vertex $v^+$ right after $w$. 

    \subsubsection*{Coloring the edges}
 
    Consider the edge $vv^+$ for some $v \le w$. From \Cref{obs:ord} and the selection of $w$ we deduce that $(\ell_{\sf r},\ell_{\sf b}) \preceq L(v)$ cannot happen. Thus, $(\ell_{\sf r},\ell_{\sf b})$ is strictly larger than $L(v)$ in at least one coordinate.  Color the edge $vv^+$ by the indices where $(\ell_{\sf r},\ell_{\sf b})$ is strictly larger than $L(v)$. Note that the edge might receive more than one color. This is not a problem, and it will come in handy later in the transitive case as we can pick the color of such edges arbitrarily. Consider the edge $v^+v$ for some $ v>w$. We cannot have $L(v) \preceq (\ell_{\sf r},\ell_{\sf b})$ by the definition of $w$. Hence, we can color the edge $v^+v$ by the indices where $L(v)$ is strictly larger than $(\ell_{\sf r},\ell_{\sf b})$. 

    \subsubsection*{Finishing the proof}
    
    Let $L^+$ be the vertex labeling under this new coloring. We have to show that $L^+(v) = L(v)$ holds for every $v \in V(H)$ and $L^+(v^+) = (\ell_{\sf r}, \ell_{\sf b})$. 
    \begin{itemize}
        \item Clearly, $L^+(v) = L(v)$ for $v \le w$, as these labels depend only on former vertices and edges. 
        
        \item Let $(x, y) \eqdef L^+(v^+)$. Suppose $v^+$ is the end point of a red path $v_{i_1}, \dots, v_{i_j}, v^+$. Since the edge $v_{i_j}v^+$ is red, our coloring implies that the first coordinate of $L(v_{i_j})$ is strictly less than $\ell_r$, and so $x \le \ell_r$. Similarly, $y \le\ell_b$. Since the minimality of $(\ell_{\sf r},\ell_{\sf b})$ implies that $(\ell_{\sf r}-1, \ell_{\sf b}), \, (\ell_{\sf r}, \ell_{\sf b}-1)$ both appear as vertex labels before $v^+$, we conclude that $L^+(v^+) = (x, y) = (\ell_{\sf r}, \ell_{\sf b})$. 
        
        \item For $v > w$, our coloring implies that $L(v) \preceq L^+(v)$. Suppose to the contrary that there is a smallest $v > w$ with $L^+(v) \ne L(v)$. Let $(x^+, y^+) \eqdef L^+(v)$ and $(x, y) \eqdef L(v)$. Assume $x^+ > x$ and a similar argument works in the $y^+ > y$ case. Then the coloring rule tells us that there is a red path of length $x^+ - 1$ going through $v^+$. Indeed, this red path has to end at $v^+v$, 
        \vspace{-0.5em}
        \begin{itemize}
            \item for otherwise it would appear to be $\cdots v^+ v' \cdots v$ for some $w < v^+ < v' < v$, resulting in $L^+(v') \ne L(v')$ due to maximal length, which contradicts the minimal assumption on $v$. 
        \end{itemize}
        \vspace{-0.5em}
        Since $v^+v$ is colored red, we have $\ell_{\sf r}<x$ by the coloring rule. On the other hand, we already know that the first coordinate of $L^+(v^+)$ is $\ell_{\sf r}$, so any red path ending with $v^+v$ has length at most $\ell_{\sf r} - 1$, and hence $x^+\le\ell_{\sf r} + 1$. Thus, we obtain $\ell_{\sf r} < x < x^+ \le \ell_{\sf r} + 1$, a contradiction. 
    \end{itemize}
    Therefore the only change among the vertex labels is the appearance of $(\ell_{\sf r},\ell_{\sf b})$. Hence, the resulting graph with a new vertex $v^+$ is $(k, \ell)$-path-free as well.
\end{proof}

\Cref{thm:sat_s} was previously proved by an analysis of the corresponding Young tableaus, which is conceptually involved (see \cite[Section 4]{DKMTWZ2021}). Extending the previous proof of the $r = 2$ case of \Cref{thm:sat_p}, we present an alternative proof using the labeling technique. 

\begin{proof}[New proof of \Cref{thm:sat_s}] 
    A $2$-coloring of an $r$-uniform complete ordered hypergraph is referred to as \textit{transitive} if for any $v_1 < v_2 < v_3$, the edges $\{v_1, v_2\}$ and $\{v_2, v_3\}$ being both red (resp. blue) implies $\{v_1, v_3\}$ being red (resp. blue). Consider a $(k, \ell)$-seq.-free sequence $v_1, \dots, v_n$ of real numbers and assume that $n < (k-1)(\ell-1)$. It induces a 2-colored ordered graph $H$ where each pair of this sequence is colored by red if it is increasing and blue if it is decreasing. Notice that $H$ is $(k-1, \ell-1)$-path-free with a transitive coloring. We will show that $H$ can be extended to a larger $(k-1, \ell-1)$-path-free ordered graph $H^+$ which is transitive as well, and it is not hard to see that this gives us an extension of the original sequence to a larger $(k, \ell)$-seq.-free sequence. (This follows, e.g., from the folklore one-to-one correspondence between permutations and transitively $2$-colored graphs: a permutation $\pi$ of $[n]$ corresponds to a transitive $2$-coloring $\chi$ where $\chi(i, j) = \textsf{red}$ if $\pi(i) < \pi(j)$ and $\chi(i, j) = \textsf{blue}$ if $\pi(i) > \pi(j)$ for all $i < j$; a transitive $2$-coloring $\chi$ gives a partial order on $[n]$ where $i \prec j$ if $\chi(i, j) = \textsf{red}$ and $j \prec i$ if $\chi(i, j) = \textsf{blue}$ for all $i < j$, and this partial order is a total order hence corresponds to a permutation $\pi$.) 
    
    We follow the steps of the proof before. Let $(\ell_{\sf r}, \ell_{\sf b}) \preceq (k, \ell)$ be a minimal missing label. We pick the position of the new vertex $v^+$ the same way. For the coloring of the new edges, we do the same procedure, except when both coordinates are strictly larger we leave the edge uncolored for the moment. That is, we color only those edges whose color is forced by the intended vertex labeling. 

    As we have seen, at this point the vertex labeling of the new coloring is already locked.  That is, we do not have long monotone paths, and no matter how we finish coloring the uncolored edges no long monotone path appears. We are left with two tasks: 
    \begin{itemize}
        \item To show that transitivity has not been violated so far. 
        \item To finish the coloring without violating transitivity. 
    \end{itemize}

    We say that a triple of vertices $v_1<v_2<v_3$ in a (partial) coloring is \emph{non-transitive} if $v_1v_2, v_2v_3$ is already colored with the same color and $v_1v_3$ is colored with the other color. Suppose, at this moment, that there is a non-transitive triple. One of them must be $v^+$, and so there are three cases. 
    \begin{itemize}
        \item If $v_1<v_2<v^+$ is non-transitive, then we can assume $v_1v^+$ is red and $v_1v_2, v_2v^+$ are blue. Since $v_1v_2$ and $v_2v^+$ are blue, the second coordinate of the vertex label increases as we go from $v_1$ to $v_2$ and then to $v^+$. Since $v_1v^+$ is red rather than uncolored, the label $L(v^+)$ cannot have a larger second coordinate than $L(v_1)$, a contradiction. 
        \item If $v^+<v_1<v_2$ is non-transitive, then a similar argument works.

        \item If $v_1<v^+<v_2$ is non-transitive, then we can assume $v_1v_2$ is red and $v_1v^+, v^+v_2$ are blue. Since $v_1v^+, v^+v_2$ are blue rather than uncolored, the first coordinates of their vertex labels decrease as we go from $v_1$ to $v^+$ and then to $v_2$. This implies that $v_1v_2$ is blue, a contradiction. 
    \end{itemize}
    Therefore there are no non-transitive triples so far.
    
    For the rest of the edges, we go through them in an arbitrary order and color them respecting transitivity. We argue that we can always pick one of the colors without creating a non-transitive triple. Suppose we are coloring an edge $AC$ where $A$ is before $C$ and we cannot color it blue. This happens exactly when one of the three cases in \Cref{fig:forcedred} appears in the coloring. 

    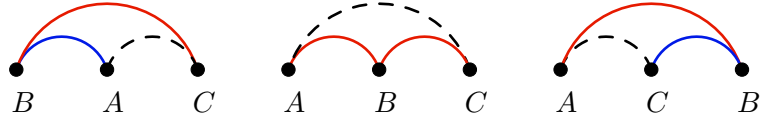
\begin{figure}[!ht]
        \centering
        \definecolor{ffqqqq}{rgb}{0.9,0.1,0.}
    \definecolor{qqqqff}{rgb}{0.,0.1,0.9}
    \scalebox{1.0}{
    \begin{tikzpicture}[line cap=round,line join=round,>=triangle 45,x=0.4cm,y=0.4cm]
    \clip(2.7667462179577775,-0.45) rectangle (29.145135513458836,4.049350671965902);
    \draw [shift={(5.5,0.5126204556506408)},line width=1.pt,color=qqqqff]  plot[domain=0.31415926535897915:2.8274333882308142,variable=\t]({1.*1.5771933363574009*cos(\t r)+0.*1.5771933363574009*sin(\t r)},{0.*1.5771933363574009*cos(\t r)+1.*1.5771933363574009*sin(\t r)});
    \draw [shift={(7.,0.025240911301281733)},line width=1.pt,color=ffqqqq]  plot[domain=0.31415926535897915:2.8274333882308142,variable=\t]({1.*3.1543866727148018*cos(\t r)+0.*3.1543866727148018*sin(\t r)},{0.*3.1543866727148018*cos(\t r)+1.*3.1543866727148018*sin(\t r)});
    \draw [shift={(8.5,0.5126204556506408)},line width=1.pt,dash pattern=on 5pt off 5pt]  plot[domain=0.31415926535897915:2.8274333882308142,variable=\t]({1.*1.5771933363574009*cos(\t r)+0.*1.5771933363574009*sin(\t r)},{0.*1.5771933363574009*cos(\t r)+1.*1.5771933363574009*sin(\t r)});
    \draw [shift={(14.5,0.5126204556506401)},line width=1.pt,color=ffqqqq]  plot[domain=0.3141592653589796:2.827433388230814,variable=\t]({1.*1.5771933363574009*cos(\t r)+0.*1.5771933363574009*sin(\t r)},{0.*1.5771933363574009*cos(\t r)+1.*1.5771933363574009*sin(\t r)});
    \draw [shift={(16.,0.025240911301280904)},line width=1.pt,dash pattern=on 5pt off 5pt]  plot[domain=0.3141592653589793:2.827433388230814,variable=\t]({1.*3.1543866727148018*cos(\t r)+0.*3.1543866727148018*sin(\t r)},{0.*3.1543866727148018*cos(\t r)+1.*3.1543866727148018*sin(\t r)});
    \draw [shift={(17.5,0.5126204556506408)},line width=1.pt,color=ffqqqq]  plot[domain=0.31415926535897915:2.8274333882308142,variable=\t]({1.*1.5771933363574009*cos(\t r)+0.*1.5771933363574009*sin(\t r)},{0.*1.5771933363574009*cos(\t r)+1.*1.5771933363574009*sin(\t r)});
    \draw [shift={(23.5,0.5126204556506417)},line width=1.pt,dash pattern=on 5pt off 5pt]  plot[domain=0.31415926535897853:2.8274333882308147,variable=\t]({1.*1.5771933363574004*cos(\t r)+0.*1.5771933363574004*sin(\t r)},{0.*1.5771933363574004*cos(\t r)+1.*1.5771933363574004*sin(\t r)});
    \draw [shift={(25.,0.025240911301281733)},line width=1.pt,color=ffqqqq]  plot[domain=0.31415926535897915:2.8274333882308142,variable=\t]({1.*3.1543866727148018*cos(\t r)+0.*3.1543866727148018*sin(\t r)},{0.*3.1543866727148018*cos(\t r)+1.*3.1543866727148018*sin(\t r)});
    \draw [shift={(26.5,0.5126204556506384)},line width=1.pt,color=qqqqff]  plot[domain=0.3141592653589806:2.8274333882308125,variable=\t]({1.*1.5771933363574013*cos(\t r)+0.*1.5771933363574013*sin(\t r)},{0.*1.5771933363574013*cos(\t r)+1.*1.5771933363574013*sin(\t r)});
    \draw (21.5,0.6) node[anchor=north west] {$A$};
    \draw (24.5,0.6) node[anchor=north west] {$C$};
    \draw (27.5,0.6) node[anchor=north west] {$B$};
    \draw (12.5,0.6) node[anchor=north west] {$A$};
    \draw (18.5,0.6) node[anchor=north west] {$C$};
    \draw (15.5,0.6) node[anchor=north west] {$B$};
    \draw (6.5,0.6) node[anchor=north west] {$A$};
    \draw (9.5,0.6) node[anchor=north west] {$C$};
    \draw (3.5,0.6) node[anchor=north west] {$B$};
    \begin{scriptsize}
    \draw [fill=black] (4.,1.) circle (2.5pt);
    \draw [fill=black] (7.,1.) circle (2.5pt);
    \draw [fill=black] (10.,1.) circle (2.5pt);
    \draw [fill=black] (13.,1.) circle (2.5pt);
    \draw [fill=black] (16.,1.) circle (2.5pt);
    \draw [fill=black] (19.,1.) circle (2.5pt);
    \draw [fill=black] (22.,1.) circle (2.5pt);
    \draw [fill=black] (25.,1.) circle (2.5pt);
    \draw [fill=black] (28.,1.) circle (2.5pt);
    \end{scriptsize}
    \end{tikzpicture}}
    \caption{Configurations that force a red edge by transitivity.}
        \label{fig:forcedred}
    \end{figure}

    Using \Cref{fig:forcedred} we can list all configurations that force an edge to be both red and blue. Up to exchanging the colors, these are listed in \Cref{fig:trans_pairs}. In each case, we indicate by dashed lines all the edges that cannot be colored without creating a non-transitive triple. 

    Each time an edge cannot be colored without breaking transitivity, there is another edge that also cannot be colored without breaking transitivity. Furthermore, the two problematic edges are always independent.  Since the already colored part of our graph is transitive this implies that if we run into any of these cases, then we have two uncolored independent edges. This is a contradiction as all uncolored edges share the vertex $v^+$.

    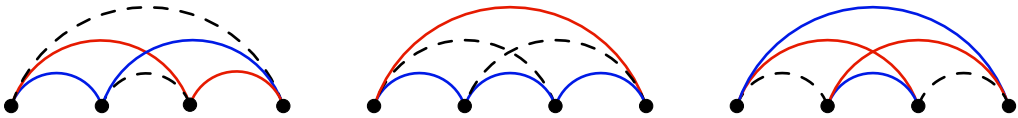
\begin{figure}[!ht]
        \centering
        \definecolor{ffqqqq}{rgb}{0.9,0.1,0.}
        \definecolor{qqqqff}{rgb}{0.,0.1,0.9}
        \scalebox{1.0}{
        \begin{tikzpicture}[line cap=round,line join=round,>=triangle 45,x=0.4cm,y=0.4cm]
            \clip(3.513890452540779,0.75) rectangle (37.6784482213951,5.246721710583301);
            \draw [shift={(5.5,0.5126204556506408)},line width=1.pt,color=qqqqff]  plot[domain=0.31415926535897915:2.8274333882308142,variable=\t]({1.*1.5771933363574009*cos(\t r)+0.*1.5771933363574009*sin(\t r)},{0.*1.5771933363574009*cos(\t r)+1.*1.5771933363574009*sin(\t r)});
            \draw [shift={(6.96319649200027,0.06151013871842712)},line width=1.pt,color=ffqqqq]  plot[domain=0.3215979510575357:2.8348720739293705,variable=\t]({1.*3.108262644940258*cos(\t r)+0.*3.108262644940258*sin(\t r)},{0.*3.108262644940258*cos(\t r)+1.*3.108262644940258*sin(\t r)});
            \draw [shift={(8.46319649200027,0.5488896830677858)},line width=1.pt,dash pattern=on 5pt off 5pt]  plot[domain=0.32926029089551695:2.8425344137673516,variable=\t]({1.*1.5311578926565925*cos(\t r)+0.*1.5311578926565925*sin(\t r)},{0.*1.5311578926565925*cos(\t r)+1.*1.5311578926565925*sin(\t r)});
            \draw [shift={(8.5,-0.4621386330480777)},line width=1.pt,dash pattern=on 5pt off 5pt]  plot[domain=0.3141592653589792:2.827433388230814,variable=\t]({1.*4.731580009072203*cos(\t r)+0.*4.731580009072203*sin(\t r)},{0.*4.731580009072203*cos(\t r)+1.*4.731580009072203*sin(\t r)});
            \draw [shift={(10.,0.02524091130128132)},line width=1.pt,color=qqqqff]  plot[domain=0.31415926535897926:2.827433388230814,variable=\t]({1.*3.1543866727148018*cos(\t r)+0.*3.1543866727148018*sin(\t r)},{0.*3.1543866727148018*cos(\t r)+1.*3.1543866727148018*sin(\t r)});
            \draw [shift={(11.44890681958776,0.5203303177690921)},line width=1.pt,color=ffqqqq]  plot[domain=0.2999178186168347:2.81319194148867,variable=\t]({1.*1.6235680023863666*cos(\t r)+0.*1.6235680023863666*sin(\t r)},{0.*1.6235680023863666*cos(\t r)+1.*1.6235680023863666*sin(\t r)});
            \draw [shift={(17.5,0.5126204556506408)},line width=1.pt,color=qqqqff]  plot[domain=0.31415926535897915:2.8274333882308142,variable=\t]({1.*1.5771933363574009*cos(\t r)+0.*1.5771933363574009*sin(\t r)},{0.*1.5771933363574009*cos(\t r)+1.*1.5771933363574009*sin(\t r)});
            \draw [shift={(19.,0.025240911301281733)},line width=1.pt,dash pattern=on 5pt off 5pt]  plot[domain=0.31415926535897915:2.8274333882308142,variable=\t]({1.*3.1543866727148018*cos(\t r)+0.*3.1543866727148018*sin(\t r)},{0.*3.1543866727148018*cos(\t r)+1.*3.1543866727148018*sin(\t r)});
            \draw [shift={(20.5,0.5126204556506401)},line width=1.pt,color=qqqqff]  plot[domain=0.3141592653589796:2.827433388230814,variable=\t]({1.*1.5771933363574009*cos(\t r)+0.*1.5771933363574009*sin(\t r)},{0.*1.5771933363574009*cos(\t r)+1.*1.5771933363574009*sin(\t r)});
            \draw [shift={(20.5,-0.4621386330480784)},line width=1.pt,,color=ffqqqq]  plot[domain=0.3141592653589793:2.827433388230814,variable=\t]({1.*4.731580009072203*cos(\t r)+0.*4.731580009072203*sin(\t r)},{0.*4.731580009072203*cos(\t r)+1.*4.731580009072203*sin(\t r)});
            \draw [shift={(22.,0.025240911301281733)},line width=1.pt,dash pattern=on 5pt off 5pt]  plot[domain=0.31415926535897915:2.8274333882308142,variable=\t]({1.*3.1543866727148018*cos(\t r)+0.*3.1543866727148018*sin(\t r)},{0.*3.1543866727148018*cos(\t r)+1.*3.1543866727148018*sin(\t r)});
            \draw [shift={(23.5,0.5126204556506417)},line width=1.pt,color=qqqqff]  plot[domain=0.31415926535897853:2.8274333882308147,variable=\t]({1.*1.5771933363574004*cos(\t r)+0.*1.5771933363574004*sin(\t r)},{0.*1.5771933363574004*cos(\t r)+1.*1.5771933363574004*sin(\t r)});
            \draw [shift={(29.5,0.5126204556506401)},line width=1.pt,dash pattern=on 5pt off 5pt]  plot[domain=0.3141592653589796:2.827433388230814,variable=\t]({1.*1.5771933363574009*cos(\t r)+0.*1.5771933363574009*sin(\t r)},{0.*1.5771933363574009*cos(\t r)+1.*1.5771933363574009*sin(\t r)});
            \draw [shift={(31.,0.025240911301281733)},line width=1.pt,color=ffqqqq]  plot[domain=0.3141592653589795:2.8274333882308147,variable=\t]({1.*3.154386672714798*cos(\t r)+0.*3.154386672714798*sin(\t r)},{0.*3.154386672714798*cos(\t r)+1.*3.154386672714798*sin(\t r)});
            \draw [shift={(32.5,0.5126204556506401)},line width=1.pt,color=qqqqff]  plot[domain=0.3141592653589796:2.827433388230814,variable=\t]({1.*1.5771933363574009*cos(\t r)+0.*1.5771933363574009*sin(\t r)},{0.*1.5771933363574009*cos(\t r)+1.*1.5771933363574009*sin(\t r)});
            \draw [shift={(32.5,-0.4621386330480784)},line width=1.pt,color=qqqqff]  plot[domain=0.3141592653589793:2.827433388230814,variable=\t]({1.*4.731580009072203*cos(\t r)+0.*4.731580009072203*sin(\t r)},{0.*4.731580009072203*cos(\t r)+1.*4.731580009072203*sin(\t r)});
            \draw [shift={(34.,0.025240911301281733)},line width=1.pt,color=ffqqqq]  plot[domain=0.31415926535897915:2.8274333882308142,variable=\t]({1.*3.1543866727148018*cos(\t r)+0.*3.1543866727148018*sin(\t r)},{0.*3.1543866727148018*cos(\t r)+1.*3.1543866727148018*sin(\t r)});
            \draw [shift={(35.5,0.5126204556506417)},line width=1.pt,dash pattern=on 5pt off 5pt]  plot[domain=0.31415926535897853:2.8274333882308147,variable=\t]({1.*1.5771933363574004*cos(\t r)+0.*1.5771933363574004*sin(\t r)},{0.*1.5771933363574004*cos(\t r)+1.*1.5771933363574004*sin(\t r)});
            \begin{scriptsize}
            \draw [fill=black] (4.,1.) circle (2.5pt);
            \draw [fill=black] (7.,1.) circle (2.5pt);
            \draw [fill=black] (9.912103311588034,1.0439790895355963) circle (2.5pt);
            \draw [fill=black] (13.,1.) circle (2.5pt);
            \draw [fill=black] (16.,1.) circle (2.5pt);
            \draw [fill=black] (19.,1.) circle (2.5pt);
            \draw [fill=black] (22.,1.) circle (2.5pt);
            \draw [fill=black] (25.,1.) circle (2.5pt);
            \draw [fill=black] (28.,1.) circle (2.5pt);
            \draw [fill=black] (31.,1.) circle (2.5pt);
            \draw [fill=black] (34.,1.) circle (2.5pt);
            \draw [fill=black] (37.,1.) circle (2.5pt);
            \end{scriptsize}
        \end{tikzpicture}
        }
        
        \caption{Transitivity problems come in pairs.}
        \label{fig:trans_pairs}
    \end{figure}

    Both tasks are done, and so the proof of \Cref{thm:sat_s} is complete. 
\end{proof}

\section{Monotone paths in ordered hypergraphs} \label{sec:hypergraph}

In this section, we prove \Cref{thm:sat_p}. Our proof is based on an enumerative result of Moshkovitz and Shapira \cite{MoshkovitzShapira2014}. For ease of notation, we only establish the result for monotone paths of the same lengths in $2$-colored hypergraphs (i.e.~$k = \ell = n$), our proof is easily generalizable to the cases when the desired monotone paths have different lengths for different (possibly more than two) colors. 

Set $\mathcal{P}_2(n) \eqdef [n]^2$. With an abuse of notations, for any $x = (x_1, x_2), \, y=(y_1, y_2) \in \mathcal{P}_2(n)$, write $x \subseteq y$ if $x_1 \leq y_1$ and $x_2 \leq y_2$. Inductively, we define $\mathcal{P}_k(n)$ for $k = 3, 4, \dots$ as follows: 
\begin{itemize}
    \item A subset $\mathcal{F}\subseteq \mathcal{P}_{k-1}(n)$ is in $\mathcal{P}_{k}(n)$ if $S\in \mathcal{F}$ implies $S'\in \mathcal{F}$ for any $S'\subseteq S$. 
\end{itemize}
In other words, $\mathcal{P}_{k}(n)$ contains those subsets of $\mathcal{P}_{k-1}(n)$ that are closed under taking subsets. If we consider the $\mathcal{P}_{k-1}(n)$-s as a poset with respect to $\subseteq$, then $\mathcal{P}_k(n)$ is the family of down-sets of $\mathcal{P}_{k-1}(n)$. We refer to this defining condition of $\mathcal{P}_{k}(n)$ as the \emph{hereditary property}.

Given any red-blue colored $r$-uniform complete ordered hypergraph $H$ on the vertex set $[N]$ containing no monochromatic monotone path of length $n$, we assign labels to each $k$-tuple with $1 \leq k \leq r-1$ in $V(H)$ as follows: 
\begin{itemize}
    \item For every $(r-1)$-tuple of vertices $v_{1}<\dots<v_{r-1}$, set $L(v_{1},\dots,v_{r-1}) \eqdef (1+\ell_{\sf r},1+\ell_{\sf b})$, where $\ell_{\sf r}$ ($\ell_{\sf b}$) is the length of the longest red (blue) monotone path ending at $v_{1},\dots,v_{r-1}$. 
    \item For $k = r-2, r-3, \dots, 1$ and every $k$-tuple of vertices $v_{1}<\dots<v_{k}$, recursively define
    \begin{equation*}
        L(v_{1},\dots,v_{k}) \eqdef \{S \in \mathcal{P}_{r-k}(n) : S \subseteq L(v_{0},v_{1},\dots,v_{k})~\text{for some}~v_{0}<v_{1}\}. 
    \end{equation*}
\end{itemize}
Obviously, the labels of $k$-tuples are elements in $\mathcal{P}_{r+1-k}(n)$ for $1 \leq k \leq r-1$. The following result, albeit not specifically stated, was proved by Moshkovitz and Shapira (Lemma~3.2 in \cite{MoshkovitzShapira2014}). Its proof essentially relies on the hereditary property of $\mathcal{P}$.

\begin{lemma} \label{ms_injective}
    For each $1\leq k\leq r-1$ and every $v_1<v_2<\dots<v_{k+1}$ in $V(H)$, we have \begin{equation*}
         L(v_1,\dots,v_{k})\not\supseteq L(v_2,\dots,v_{k+1}). 
    \end{equation*}
    In particular, $u \neq v$ implies $L(u) \neq L(v)$, and so $\ram^{(r)}_{\sf p}(n) \leq |\mathcal{P}_{r}(n)|$. 
\end{lemma} 

We use abbreviated notations $\ram_{\sf p}^{(r)}(n) \eqdef \ram_{\sf p}^{(r)}(n, n)$ and $\sat_{\sf p}^{(r)}(n) \eqdef \sat_{\sf p}^{(r)}(n, n)$. Moshkovitz and Shapira also provided constructions (Lemma~3.4 in \cite{MoshkovitzShapira2014}) to prove that
\begin{equation}\label{ms_bijective}
    \ram^{(r)}_{\sf p}(n) = |\mathcal{P}_{r}(n)|.
\end{equation}

Now we are ready to state the following result. Together with \eqref{ms_bijective} it implies \Cref{thm:sat_p}. 

\begin{theorem}\label{hypergraph}
    For any integers $r \geq 2$ and $n \geq 1$, we have $\sat^{(r)}_{\sf p}(n) = |\mathcal{P}_{r}(n)|$. 
\end{theorem}

We follow the steps from the proof in \Cref{sec:warmup}. If we have less than $|\mathcal{P}_{r}(n)|$ vertices, then after creating the vertex labeling, one of the possible labels is missing. We will extend the hypergraph and its coloring so that the new vertex receives one of the missing labels while the labels of other vertices do not change. 

\begin{proof}[Proof of \Cref{hypergraph}]
    Let $H$ be a red-blue colored $r$-uniform complete ordered hypergraph $H$ on the vertex set $V(H)=[N]$ containing no monochromatic monotone path of length $n$. Suppose $N < |\mathcal{P}_{r}(n)|$. Our goal is to show that $H$ is not saturated. Let $L$ be the labeling for $H$ as described earlier in this section. Since $N < |\mathcal{P}_{r}(n)|$, there exists a missing label $M \in \mathcal{P}_{r}(n)$ with $M \neq L(v)$ for each vertex $v \in V(H)$.
    
    \subsubsection*{Picking the position}
    Let $w$ be the largest vertex in $V(H)$ such that $L(w) \subseteq M$. Such a $w$ exists because $L(1)$ takes the minimum value of $\mathcal{P}_{r}(n)$ with respect to ``$\subseteq$''. We construct a new $r$-uniform complete ordered hypergraph $H^+$ by adding a new vertex $v^+$ into $V(H)$ right after $w$ and keeping the colors of the edges originally in $H$. If suffices to show that we can color the additional edges of $H^+$ without creating monochromatic monotone paths of length $n$.

    \subsubsection*{Coloring the edges}
    We begin with assigning potential labels for the hypergraph $H^+$. For each $1 \leq k \leq r-1$ and $k$-tuple $v_1<\dots<v_k$ in $V(H^+)$, assign a potential label $\widetilde{L}(v_1,\dots,v_k) \in \mathcal{P}_{r+1-k}(n)$ with
    \begin{itemize}
        \item[(i)] $\widetilde{L}(v_1,\dots,v_k) \not\supseteq \widetilde{L}(v_2,\dots,v_{k+1})$ for all $v_1<\dots<v_{k+1}$ in $V(H^+)$, and
        \item[(ii)] $\widetilde{L}(v_1,\dots,v_k) = L(v_1,\dots,v_k)$ for all $v_1<\dots<v_k$ from the original $V(H)$. 
    \end{itemize}
    Define $\widetilde{L}(v^+) \eqdef M$, the missing label, and $\widetilde{L}(v) \eqdef L(v)$ for all $v \in V(H)$. By \Cref{ms_injective} and our construction of $H^+$, the conditions (i) and (ii) are satisfied for $k=1$. 
    
    Inductively, suppose the potential labels for all $(k-1)$-tuples have been assigned, and we are in the position to define $\widetilde{L}(v_1,\dots,v_k)$ for every $v_1<\dots<v_k$ in $V(H^+)$. Due to (ii), we have to set $\widetilde{L}(v_1,\dots,v_k) \eqdef L(v_1,\dots,v_k)$ if this $k$-tuple comes from $V(H)$. For other $k$-tuples containing $v^+$, we define $\widetilde{L}(v_1,\dots,v_k)$ as an arbitrary fixed element from $\widetilde{L}(v_2,\dots,v_k) \setminus \widetilde{L}(v_1,\dots,v_{k-1})$. Since condition (i) is satisfied for $k-1$, such an element always exists. In fact, for any $v_1 < \dots < v_k$ in $V(H)$, we also have  
    \[
    \widetilde{L}(v_1, \dots, v_k) \in \widetilde{L}(v_2, \dots, v_k) \setminus \widetilde{L}(v_1, \dots, v_{k-1}).
    \]
    Indeed, due to $\widetilde{L} = L$ on $V(H)$, the only thing we need to check is $L(v_1, \dots, v_k) \not\in L(v_1, \dots, v_{k-1})$. And this follows from the definition of $L$ and Lemma~\ref{ms_injective}.
    
    We need to check that condition (i) is satisfied for $k$. Suppose for the sake of contradiction that $\widetilde{L}(v_2,\dots,v_{k+1})\subseteq \widetilde{L}(v_1,\dots,v_k)$ for $v_1<\dots<v_{k+1}$. Then by hereditary property, 
    \begin{align*}
    \widetilde{L}(v_1, \dots, v_k) \in \widetilde{L}(v_2, \dots, v_k) \implies \widetilde{L}(v_2, \dots, v_{k+1}) \in \widetilde{L}(v_2, \dots, v_k), 
    \end{align*}
    which contradicts the fact that $\widetilde{L}(v_2, \dots, v_{k+1}) \in \widetilde{L}(v_3, \dots, v_{k+1}) \setminus \widetilde{L}(v_2, \dots, v_k)$. So, (i) holds. We conclude that the potential labels $\widetilde{L}$ can be recursively assigned.
    
    Now, we can color the new edges using $\widetilde{L}$, the potential labels. For any edge $v_1 \dotsb v_r \in E(H^+)$ with $v^+ \in \{v_1,\dots,v_r\}$, the condition (i) implies that $\widetilde{L}(v_1,\dots,v_{r-1}) \not\supseteq \widetilde{L}(v_2,\dots,v_{r})$ as elements in $\mathcal{P}_{2}(n) = [n]^2$. So, at least one coordinate of $\widetilde{L}(v_2,\dots,v_r)$ is larger than that of $\widetilde{L}(v_1,\dots,v_{r-1})$. Color $v_1 \dotsb v_r$ red if the first coordinate is larger, and blue if the second coordinate is larger. For edges that are both red and blue, we arbitrarily assign a color. 

    \subsubsection*{Finishing the proof}
    
    We show that $H^+$ contains no monochromatic monotone path of length $n$. For every $(r-1)$-tuple $v_1<\dots<v_{r-1}$ in $V(H^+)$, set $L^+(v_1,\dots,v_{r-1}) \eqdef (1+\ell_{\sf r}, 1+\ell_{\sf b})$ where $\ell_{\sf r}$ (resp.~$\ell_{\sf b}$) is the length of the longest red (resp.~blue) monotone path in $H^+$ ending at $v_1,\dots,v_{r-1}$. We shall prove that
    \begin{equation}\label{hypergraph_eq1}
        L^+(v_1,\dots,v_{r-1}) \subseteq \widetilde{L}(v_1,\dots,v_{r-1}) ~\text{for all}~ v_1<\dots<v_{r-1} ~\text{in}~ V(H^+).
    \end{equation}
    Since $\widetilde{L}(v_1,\dots,v_{r-1})$ takes its value in $\mathcal{P}_{2}(n) = [n]^2$, \Cref{hypergraph} follows from \eqref{hypergraph_eq1}. 

    For a contradiction, suppose \eqref{hypergraph_eq1} is violated by some $(r-1)$-tuple in $V(H^+)$. Let $v_1 < \dots < v_{r-1}$ be the smallest such tuple under the lexicographic order. According to the definition of $L^+$, this violation is witnessed by a monochromatic (say red) monotone path $P$ ending at $v_1, \dots, v_{r-1}$. Let $e \eqdef v_0v_1 \dotsb v_{r-1}$ be the last edge and $\ell$ be the length of this red path. Then $1 + \ell$ is larger than the first coordinate of $\widetilde{L}(v_1, \dots, v_{r-1})$. The minimum assumption on $v_1, \dots, v_{r-1}$ implies that 
    \begin{equation}\label{hypergraph_eq2}
        L^+(v_0,\dots,v_{r-2}) \subseteq \widetilde{L}(v_0,\dots,v_{r-2}).
    \end{equation}
    We then separate our indirect proof into two cases: 
    \begin{itemize}
        \item If $v^+ \in e$, then $\widetilde{L}(v_1,\dots,v_{r-1})$ has a larger first coordinate than $\widetilde{L}(v_0,\dots,v_{r-2})$ since $e$ is red, and so $\ell$ is larger than the first coordinate of $\widetilde{L}(v_0,\dots,v_{r-2})$. On the other hand, notice that $P \setminus \{e\}$ is a red monotone path of length $\ell-1$ ending at $v_0,v_1,\dots,v_{r-2}$. This means the first coordinate of $L^+(v_0,\dots,v_{r-2})$ is at least $\ell$, a contradiction to \eqref{hypergraph_eq2}.
        
        \item If $v^+ \notin e$, then the vertices $v_0,v_1\dots,v_{r-1}$ are all in $V(H)$. From condition (ii) of the potential labeling and \eqref{hypergraph_eq2} we obtain $L(v_0,\dots,v_{r-2}) = \widetilde{L}(v_0,\dots,v_{r-2}) \supseteq L^+(v_0,\dots,v_{r-2})$. Again, the path $P\setminus\{e\}$ implies that the first coordinate of $L^+(v_0,\dots,v_{r-2})$ is at least $\ell$, and so there is a red monotone path in $H$ ending at $v_0,\dots,v_{r-2}$ of length at least $\ell-1$. Together with $e$, we have a red monotone path in $H$ ending at $v_1,\dots,v_{r-1}$ of length at least $\ell$. It follows that the first coordinate of $L(v_1,\dots,v_{r-1})$ is at least $1+\ell$, which is larger than the first coordinate of $\widetilde{L}(v_1,\dots,v_{r-1})$. This contradicts condition (ii) of the potential labeling. \qedhere
    \end{itemize}
\end{proof}

\section{Saturation for cups-versus-caps} \label{sec:cupcap}

In this section, we study the saturation problem for cups and caps. Recall that the definitions directly imply $\sat_{\sf c}(k, \ell) \le \ram_{\sf c}(k, \ell)$ for any integers $k, \ell$. By reflecting over any horizontal line, it is easily seen that $\sat_{\sf c}(k, \ell) = \sat_{\sf c}(\ell, k)$. 

\smallskip

The study begins with a basic property of $(k, \ell)$-cup-cap-saturated sets. 

\begin{proposition} \label{prop:cc_satexist}
    Let $k, \ell \ge 2$ be integers. If $P$ is a $(k, \ell)$-cup-cap-saturated set, then 
    \begin{itemize}
        \item there exist $k-1$ points of $P$ that form a $(k-1)$-cup, and
        \item there exist $\ell-1$ points of $P$ that form an $(\ell-1)$-cap. 
    \end{itemize}
\end{proposition}

\begin{proof}
    Due to the symmetry, it suffices to prove the first statement. Assume, for the sake of contradiction, that $P \subset \mathbb{R}^2$ is a $(k, \ell)$-cup-cap-saturated set without any $(k-1)$-cup subset. Choose an arbitrary point $q$ such that
    \begin{itemize}
        \item the $x$-coordinate of $q$ is bigger than every point from $P$, and 
        \item the $y$-coordinate of $q$ is big enough so that $q$ is above every line spanned by $P$. 
    \end{itemize}
    Then $P \cup \{q\}$ is generic and $p_1, p_2, q$ form a $3$-cup for any choice of $p_1, p_2 \in P$. This implies that $P \cup \{q\}$ is $(k, \ell)$-cup-cap-free, which contradicts the saturation property of $P$. 
\end{proof}

We work out some values of $\sat_{\sf c}(k, \ell)$ where at least one of $k$ and $\ell$ is small. 

\begin{proposition}\label{thm:sat_csmall}
    Let $\ell$ be a positive integer. \begin{itemize}
        \item $\sat_{\sf c}(1, \ell) = 0 = \ram_{\sf c}(1, \ell)$ for any $\ell \ge 1$. 
        \item $\sat_{\sf c}(2, \ell) = 1 = \ram_{\sf c}(2, \ell)$ for any $\ell \ge 2$.
        \item $\sat_{\sf c}(3, \ell) = \ell-1 = \ram_{\sf c}(3, \ell)$ for any $\ell \ge 3$.
        \item $\sat_{\sf c}(4, 4) = 6 = \ram_{\sf c}(4, 4)$.
    \end{itemize}
\end{proposition}

\begin{proof}
    The facts $\sat_{\sf c}(1, \ell) = 0$ and $\sat_{\sf c}(2, \ell) = 1$ are strightforward corollaries of the definitions. From \Cref{prop:cc_satexist} we deduce that $\ram_{\sf c}(3, \ell) = \ell-1 \le \sat_{\sf c}(3, \ell)$, and so $\sat_{\sf c}(3, \ell) = \ell-1$. 

    Obviously, $\sat_{\sf c}(4, 4) \ge 3$. Let $P = \{(x_1, y_1), \dots, (x_m, y_m)\}$ be a generic $(4, 4)$-cup-cap-free set with $x_1 < \dots < x_m$ and $3 \le m \le 5$. Consider $q \eqdef (\frac{x_2+x_3}{2}, y)$ such that $P \cup \{q\}$ is generic. If $q$ is very high up (i.e., above any line formed by two points of $P$) and $P \cup \{q\}$ is not $(4, 4)$-cup-cap-free, then the only case to make this happen is that $(x_3, y_3), (x_4, y_4), (x_5, y_5)$ form a $3$-cup. Indeed, $q$ has to be part of some $4$-cup or $4$-cap in $P \cup \{q\}$, yet it cannot make a $4$-cap since the $x$-coordinate of $q$ is between $x_2$ and $x_3$. It follows that $P \cup \{q\}$ is $(4, 4)$-cup-cap-free if $y$ is chosen to make $q$ very low below (i.e., below any line formed by two points of $P$). We conclude that $\sat_{\sf c}(4, 4) \ge 6 = \ram_{\sf c}(4, 4)$, and so $\sat_{\sf c}(4, 4) = 6$. 
\end{proof}

The following result gives the first example with $\sat_{\sf c}(k, \ell) < \ram_{\sf c}(k, \ell)$, which is crucial for our proof of the upper bound in \Cref{thm:sat_c}. 

\begin{theorem}\label{thm:sat45}
    We have $\sat_{\sf c}(4, 5) = 8 < 10 = \ram_{\sf c}(4, 5)$.
\end{theorem}

Before proving \Cref{thm:sat45}, we need some preparation. For any point $p = (a, b)$ in the plane, set $x(p) \eqdef a$. When denoting by $\overline{p_1 \dotsb p_m}$ an $m$-cup or cap, we implicitly assume $x(p_1) < \dots < x(p_m)$. For any $m$-cup or cap $\overline{p_1 \dotsb p_m}$, we define its \textit{shadow} as the open interval $\bigl(x(p_1), x(p_m)\bigr)$ in $\mathbb{R}$. 

\begin{lemma} \label{shadow}
    Let $P \subset \mathbb{R}^2$ be a $(k, \ell)$-cup-cap-saturated set with $k \ge 4$ and $\ell \ge 4$. Then $P$ contains two $(k-1)$-cups with disjoint shadows and two $(\ell-1)$-caps with disjoint shadows. 
\end{lemma}

\begin{proof}
    We show the existence of two $(k-1)$-cups with disjoint shadows, and there exist two such $(\ell-1)$-caps for similar reasons. For a contradiction, suppose there are no $(k-1)$-cups with disjoint shadows. Then the intersection of all shadows from $(k-1)$-cups in $P$ is nonempty, and so there exists some $x_* \in \mathbb{R}$ in this intersection with $x(p) \neq x^*$ for all $p \in P$. Consider a generic point $q \eqdef (x^*, y)$, where $y$ is some sufficiently large number such that, together with $q$, 
    \begin{itemize}
        \item every $p_i, p_j \in P$ with $x(p_i) < x(p_j) < x(q)$ or $x(q) < x(p_i) < x(p_j)$ form a $3$-cup, and
        \item every $p_i, p_j \in P$ with $x(p_i) < x(q) < x(p_j)$ form a $3$-cap. 
    \end{itemize}
    We argue that $P \cup \{q\}$ is $(k, \ell)$-cup-cap-free, a contradiction to the hypothesis that $P$ is saturated. 
    Indeed, the construction shows that $q$ cannot be part of any $4$-cap in $P \cup \{q\}$, and so $P \cup \{q\}$ contains no $\ell$-cap. If $q$ is part of some $k$-cup $\overline{a_1 \dotsb a_s q b_1 \dotsb b_t}$ in $P \cup \{q\}$, then $s \ge 1$ and $t \ge 1$, for otherwise $x^*$ would lie outside the shadow of the $(k-1)$-cup $\overline{a_1 \dotsb a_s b_1 \dotsb b_t}$. However, this contradicts the fact that $\overline{a_1qb_1}$ is a $3$-cap, and so $P \cup \{q\}$ contains no $k$-cup. 
\end{proof}

\begin{proof}[Proof of \Cref{thm:sat45}]
First, we show that $\sat_{\sf c}(4, 5) \leq 8$. It suffices to construct a generic $8$-point set that is $(4, 5)$-cup-cap-saturated. Consider the following $8$ points as $P$: 
\begin{equation*}
    (-60, 40),~(-40, 20),~(-20, 16),~(0, 10),~(5, -50),~(15, -40),~(25, -40),~(125, -230).
\end{equation*}
It is straightforward to verify that $P$ is $(4, 5)$-cup-cap-free. We use a computer program to check the saturation property. This program first computes all the lines generated by pairs in $P$ and the vertical lines through every point of $P$ and works out all the regions of $\mathbb{R}^2$ enclosed by these lines. Then the program picks a point from each region, adds it to the point set $P$, and verifies that the new set contains a $4$-cup or a $5$-cap. See \Cref{sec:code} for our supplementary code.

We remark that one can also check by hand that $P$ is saturated, albeit this process is slightly tedious. For example, name the points of $P$ as $p_1,p_2,\dots,p_8$ with increasing $x$-coordinates and we argue that any new point $q$ added into the triangle $\Delta$ spanned by $p_2, p_3, p_4$ is part of a $4$-cup or $5$-cap. Consider the vertical line at $p_3$, the line $p_3p_6$, and the line $p_1p_3$. These three lines cut $\Delta$ into four subtriangles $\Delta_1,\Delta_2,\Delta_3,\Delta_4$ from left to right. See \Cref{fig:8points} for an illustration. We can check through elementary algebra the intersection between the lines $p_1p_3$ and $p_6p_8$ are outside $\Delta$, and so $\Delta_3$ is below the line $p_6p_8$. It follows that $\overline{p_1p_2qp_3}$ is a $4$-cup if $q\in \Delta_1$, $\overline{p_3qp_6p_7}$ is a $4$-cup if $q \in \Delta_2$, $\overline{p_2p_3qp_6p_8}$ is a $5$-cap if $q \in \Delta_3$, and $\overline{p_1p_3qp_4}$ is a $4$-cup if $q\in \Delta_4$. 

\begin{figure}[!ht]
    \centering
    \definecolor{qqqqff}{rgb}{0.,0.1,0.9}
\definecolor{ffqqqq}{rgb}{0.9,0.1,0.}
\scalebox{1.0}{
\begin{tikzpicture}[line cap=round,line join=round,>=triangle 45,x=1.4cm,y=1.4cm]
\clip(4.003117423407316,1.2) rectangle (11.741941955611567,5.783981863045565);
\draw [line width=1.5pt] (7.974368387419457,5.433757687135593)-- (7.974368387419457,1.3553406708936568);
\draw [line width=1.5pt] (7.974368387419457,3.773274724328437)-- (11.403015333763188,2.1574670092681094);
\draw [line width=1.5pt] (11.403015333763188,2.1574670092681094)-- (4.217770950577939,4.1910267403582715);
\draw [line width=1.5pt] (4.217770950577939,4.1910267403582715)-- (7.974368387419457,3.773274724328437);
\draw [line width=1.5pt,dash pattern=on 3pt off 3pt] (9.211289845810455,5.433757687135593)-- (10.634781657573571,1.2536626843391487);
\draw [line width=1.5pt,color=ffqqqq] (6.375603776345836,4.9253677543630525)-- (11.566611099167737,1.1846523832205462);
\draw [line width=1.5pt,color=qqqqff] (9.932150211905967,1.2582545737640638)-- (6.804910830687094,5.275591930273025);
\begin{scriptsize}
\draw (9.25,5.569328335874936) node[anchor=north west] {$p_6p_8$};
\draw (11.25,2.5) node[anchor=north west] {$p_4$};
\draw (4.05,4.5) node[anchor=north west] {$p_2$};
\draw (7.98,4.05) node[anchor=north west] {$p_3$};
\draw (5.9,4.947962862486276) node[anchor=north west] {$p_1p_3$};
\draw (6.8,5.5) node[anchor=north west] {$p_3p_6$};
\draw (6.7823157225638795,3.8069099022634627) node[anchor=north west] {$\Delta_1$};
\draw (7.95,3.47) node[anchor=north west] {$\Delta_2$};
\draw (8.51,3.25) node[anchor=north west] {$\Delta_3$};
\draw (9.25,3.1) node[anchor=north west] {$\Delta_4$};
\draw [fill=black] (7.974368387419457,3.773274724328437) circle (1.5pt);
\draw [fill=black] (11.403015333763188,2.1574670092681094) circle (1.5pt);
\draw [fill=black] (4.217770950577939,4.1910267403582715) circle (1.5pt);
\end{scriptsize}
\end{tikzpicture}
}
\caption{Around the triangle $\Delta$ spanned by $p_2,p_3,p_4$. The line $p_6p_8$ does not divide $\Delta_4$.}
\label{fig:8points}
\end{figure}
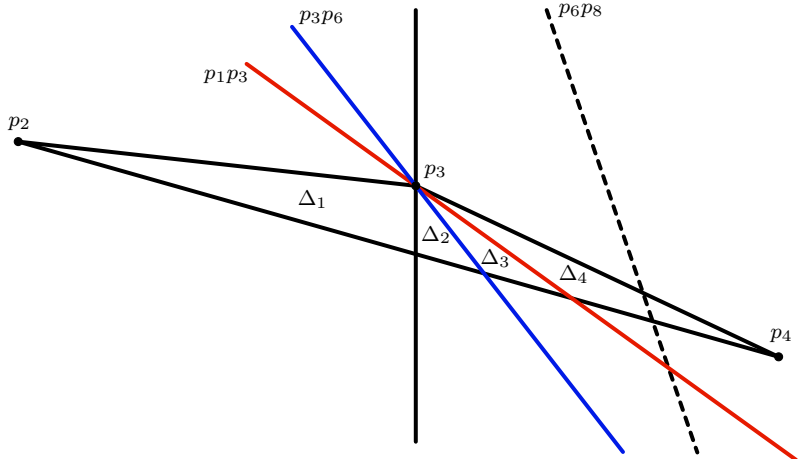

Next, we prove that $\sat_{\sf c}(4,5)\geq 8$. Let $P$ be a $(4, 5)$-cup-cap-saturated point set, we argue that $|P| \geq 8$. By \Cref{shadow}, we can find two $4$-caps in $P$, $\overline{p_1p_2p_3p_4}$ and $\overline{p_5p_6p_7p_8}$ with $x(p_4) \le x(p_5)$, whose shadows are disjoint. If the points $p_4 \neq p_5$, then $|P| \geq 8$ and the proof is done. Assume that $p_4 = p_5 = p$. Then the shadow of any $3$-cup within $P_- \eqdef \{p_1, p_2, p_3, p, p_6, p_7, p_8\}$ must contain $x(p)$. By \Cref{shadow}, there are two $3$-cups with disjoint shadows in $P$. So, at least one point from $P$ does not belong to $P_-$, and hence $|P| \geq 8$.
\end{proof}

We shall use \Cref{thm:sat45} to upper bound $\sat_{\sf c}(k,\ell)$. It follows from the next lemma.

\begin{lemma}\label{lem:cupcombine}
    For any $k,\ell \geq 3$, we have $\sat_{c}(k,\ell)\leq \sat_{c}(k-1,\ell)+\sat_{c}(k,\ell-1)$. 
\end{lemma}

\begin{proof}
    Suppose $P_{k-1, \ell} \subset \mathbb{R}^2$ and $P_{k, \ell-1} \subset \mathbb{R}^2$ are $(k-1, \ell)$- and $(k, \ell-1)$-cup-cap-saturated, respectively. We construct $P$ with $|P| = |P_{k-1, \ell}| + |P_{k, \ell-1}|$ that is $(k, \ell)$-cup-cap-saturated. 
    
    We begin with a fixed vertical line $h$. Put a translated copy $A$ of $P_{k-1, \ell}$ to the left of $h$, and a translated copy $B$ of $P_{k, \ell-1}$ to the right of $h$. Vertically shift $B$ to somewhere very high so that every point in $A$ is below all lines spanned by $B$, and every point in $B$ is above all lines spanned by $A$. This implies that
    \begin{itemize}
        \item any one point from $A$ and two points from $B$ form a $3$-cap, and
        \item any two points from $A$ and one point from $B$ form a $3$-cup. 
    \end{itemize}
    These conditions guarantee that $P \eqdef A \cup B$ is $(k, \ell)$-cup-cap-free. 
    
    We claim that $P$ is $(k, \ell)$-saturated. It suffices to disprove that some point $q \notin P$ makes a generic $P \cup \{q\}$ that is $(k, \ell)$-cup-cap-free. Without loss of generality, assume that $q$ lies to the left of $h$ (possibly on $h$). Since $A$ is $(k-1, \ell)$-cup-cap-saturated and $P \cup \{q\}$ is $(k, \ell)$-cup-cap-free, the point $q$ has to be part of a $(k-1)$-cup $C_1 = \overline{a_1 \dots a_{k-1}}$ in $A \cup \{q\}$, where $a_{k-1}$ (possibly be $q$) lies to the left of $h$ (possibly on $h$). Since $B$ is $(k, \ell-1)$-cup-cap-saturated and $P \cup \{q\}$ is $(k, \ell)$-cup-cap-free, the point $a_{k-1}$ has to be part of an $(\ell-1)$-cap $C_2 = \overline{b_1 \dots b_{\ell-1}}$ where $b_1 = a_{k-1}$. So, the cup $C_1$ and the cap $C_2$ share $a_{k-1} = b_1$, and hence $C_1$ can be extended by $b_2$ or $C_2$ can be extended by $a_{k-2}$, which contradicts the fact that $P \cup \{q\}$ is $(k, \ell)$-cup-cap-free. This completes the proof.
\end{proof}

\begin{proof}[Proof of Theorem~\ref{thm:sat_c}]
    The upper bound follows from solving (for $k, \ell \ge 1$) the recursion
    \[
    \sat_{\sf c}(k,\ell) \leq \sat_{c}(k-1,\ell) + \sat_{c}(k,\ell-1)
    \]
    from \Cref{lem:cupcombine}, with initial values given by \Cref{thm:sat_csmall} and \Cref{thm:sat45}. 
    
    Notice that any two cups (resp.~caps) with disjoint shadows share no more than one point. Since any cup and any cap share no more than two points, from \Cref{shadow} we deduce that
    \begin{equation*}
        \sat_{\sf c}(k,\ell) \geq 2(k-1) + 2(\ell-1) - 1 - 1 - 2 - 2 - 2 - 2 = 2k + 2\ell - 14. \qedhere
    \end{equation*}
\end{proof}

We remark that it is possible to improve the number ``$14$'' in the lower bound by counting the intersections of those cups and caps with disjoint shadows more carefully. 

\section{Saturation for convex polygons}\label{sec:gon}

This section is devoted to the proof of \Cref{thm:sat_g}. Recall that Erd\H{o}s and Szekeres established the result $\ram_{\sf g}(n) \ge 2^{n-2}$ by taking a union of $(i+1, n+1-i)$-cup-cap-free sets appropriately along a convex curve. Roughly speaking, we shall similarly establish the lower bound on $\sat_{\sf g}(n)$ via replacing each $(i+1, n+1-i)$-cup-cap-free set before by an $(i+1, n+1-i)$-cup-cap-saturated set. 

For a technical issue in \Cref{lem:rotate}, we need the following notion: a planar point set is called \textit{very generic} if it is in general position, and its members together with the intersection points of the lines spanned by it all have distinct $x$-coordinates. The main result of this section is the following: 

\begin{proposition}\label{thm:gconstruction}
    For any positive integer $n\geq 3$, suppose that $P_i$ is a very generic planar point set that is $(i+1,n+1-i)$-cup-cap-saturated for each $1\leq i< n$, then there exists an $n$-gon-saturated set $P$ such that $|P| = \sum_{i=1}^{n-1} |P_i|$.
\end{proposition}

We point out that our upper bounds in previous sections, \Cref{thm:sat_csmall}, \Cref{thm:sat45,thm:sat_c}, all give us $(k,\ell)$-cup-cap-saturated sets that have the claimed size and are very generic, although the very generic property is not stated. Together with \Cref{thm:gconstruction} and \eqref{eq:ES_construction}, these upper bounds imply that for every $n \ge 7$, 
\[
\sat_{\sf g}(n) \le \sum_{i=1}^{n-1} \binom{n-2}{i-1} - 2\sum_{i=1}^{n-1} \binom{n-6}{i-3} = 2^{n-2} - 2^{n-5} = \frac{7}{8} \cdot 2^{n-2} \le \frac{7}{8} \ram_{\sf g}(n). 
\]
This concludes the proof of \Cref{thm:sat_g}. 

In \Cref{thm:gconstruction}, the $n = 3, \, n = 4$ cases are easy. Indeed, every $2$-point set is $3$-gon-saturated, and every $4$-point set not in convex position is $4$-gon saturated. For the rest of this section, we implicitly assume $n \ge 5$. To prove \Cref{thm:gconstruction}, our strategy is to place an appropriate copy of a $(i+1,n+1-i)$-cup-cap-saturated set around the point $(i, i^2)$ for $i = 1, \dots, n$. This is motivated by the original Erd\H{o}s--Szekeres construction~\cite{ErdosSzekeres1961} showing $\ram_{\sf g}(n) \ge 2^{n-2}$. We remark that a $(k,\ell)$-cup-cap-free point set of size $\ram_{\sf c}(k,\ell)$ is always $(k,\ell)$-cup-cap-saturated. Furthermore, a small rotation can be applied such that it becomes very generic. In this way, our \Cref{thm:gconstruction} implies that the construction of Erd\H{o}s and Szekeres is also $n$-gon-saturated.

\smallskip

We begin with some preparation. For two points $p, q \in \mathbb{R}^2$ with $x(p) \neq x(q)$, denote by
\begin{itemize}
    \item $\lin(p)$ the vertical line through $p$,
    \item $\lin(p, q)$ the unique line through $p$ and $q$, 
    \item $\ray(p, q)$ the ray emanating from $p$ through $q$, 
    \item $\slp(p, q)$ the slope of $\lin(p, q)$. 
\end{itemize}

We say that a point set $P$ is \textit{$(k,\ell;\varphi)$-cup-cap-saturated} if $P$ will be $(k, \ell)$-cup-cap-saturated after a rotation of any angle $\theta \in (-\varphi,\varphi)$ with arbitrary center. The next result shows that, given $\varphi$, if we flatten a $(k, \ell)$-cup-cap-saturated set enough, then it becomes $(k,\ell;\varphi)$-cup-cap-saturated.

\begin{lemma}\label{lem:rotate}
    For a very generic $(k, \ell)$-cup-cap-saturated set $P_0$ and a positive real number $\varphi< \pi/2$, there exists some sufficiently small $\delta>0$ such that the ``flattening'' map $\sigma \colon (x, y) \mapsto (x, \delta y)$ produces a set $P \eqdef \sigma(P_0)$ which is $(k, \ell; \varphi)$-cup-cap-saturated.
\end{lemma}

\begin{proof}
    Let $Q_0$ be the set of intersection points of lines spanned by pairs in $P_0$. Let $\sigma, P$ be defined as in our statement and $Q \eqdef \sigma(Q_0)$. By the nature of ``flattening'', $P$ is also $(k, \ell)$-cup-cap-saturated and $Q$ consists of intersection points of lines spanned by $P$.
    
    For each $p\in P$, denote by $K_p$ the cone swept over by rotating $\lin(p)$ around $p$ with an angle from $[-\varphi,\varphi]$. We take $\delta$ to be sufficiently small such that
    \begin{itemize}
        \item[(i)] $x \not\in K_p$ for all $p \in (P \cup Q) \setminus \{x\}$. 
        \item[(ii)] $K_{p_1} \cap K_{p_2} \cap \lin(p_3, p_4) = \varnothing$ for any distinct $p_1, p_2 \in P$ and any distinct $p_3, p_4 \in P$. 
    \end{itemize}
    The value of $\delta$ satisfying these two conditions is obviously an open set of $\mathbb{R}$, and the very generic condition implies that $\delta = 0$ is in this set. Hence, there is indeed such a choice of $\delta$. 

    Now we prove that $P$ is $(k, \ell; \varphi)$-cup-cap-saturated. Take an arbitrary $\theta\in (-\varphi,\varphi)$ and let $\tau$ be a rotation of angle $\theta$ with an arbitrary center. Consider the image of the original vertical line through $p$, denoted as $\tau(\lin(p))$, it is now a non-vertical line through $\tau(p)$. We use $F_{\tau(p)}$ to denote the closed cone between the two lines $\lin(\tau(p))$ and $\tau(\lin(p))$. The vertex of this cone $F_{\tau(p)}$ is $\tau(p)$ and the angle of this cone is $\theta$. 

    Let $p^+$ be an arbitrary new generic point added into $\tau(P)$. Consider all the lines spanned by $\tau(P)$ together with all the vertical lines passing through points in $\tau(P)$. These lines cut the plane into polygonal cells and we denote the cell containing $p^+$ as $\Omega$. We claim that
    \begin{equation}\label{eq:rotate}
        \text{$\Omega$ contains a point $q^+ \not\in \bigcup_{p \in P} F_{\tau(p)}$.}
    \end{equation}
    Suppose this claim is given for now. Consider the point $\tau^{-1}(q^+)$ added into $P$ which is $(k, \ell)$-cup-cap-saturated, and let $C$ be a resulting $k$-cup or an $\ell$-cap containing $\tau^{-1}(q^+)$. Notice that $\tau$ will not change the relative order between $x \bigl( \tau^{-1}(q^+) \bigr)$ and $x(p)$ for any $p\in C$, because otherwise $q^+$ would be inside $F_{\tau(p)}$. Similarly, the relative order between $x(p_1)$ and $x(p_2)$ will not change for any $p_1,p_2\in C$, otherwise we can deduce $p_1\in K_{p_2}$ violating condition (i). This means $\tau(C)$ is still a $k$-cup or an $\ell$-cap containing $q^+$. Since $p^+$ and $q^+$ are in the same polygonal cell, $p^+$ together with $\tau(C) \setminus \{q^+\}$ forms a $k$-cup or an $\ell$-cap, meaning that $\tau(P)$ is $(k,\ell)$-cup-cap-saturated.

    It suffices to verify \eqref{eq:rotate} then. Our proof is separated into two cases: 
    \begin{itemize}
        \item If the boundary vertices of $\Omega$ are not all in $\bigcup \lin(\tau(p))$, then one of them, denoted as $v$, must be an intersection point of two lines spanned by $\tau(P)$. Since rotation is an affine transformation, $v = \tau(q)$ for some $q \in Q$. If $v \in F_{\tau(p)}$ for some $p \in P$, then $q \in K_p$ which contradicts condition (i). Hence we have $v \not\in \bigcup F_{\tau(p)}$ and we can find $q^+$ very close to $v$ within the interior of $\Omega$.

        \item If the boundary vertices of $\Omega$ are all in $\bigcup \lin(\tau(p))$, then consider a boundary segment $s$ of $\Omega$ that is not vertical (its existence is easy to argue). Let $p_1, p_2$ be two points in $P$ such that the endpoints of $s$ belongs to $\lin(\tau(p_1))$ and $\lin(\tau(p_2))$ respectively. The means the endpoints of $s$ are in $F_{\tau(p_1)}$ and $F_{\tau(p_2)}$ respectively. If $s$ is covered by these two closed sets $F_{\tau(p_1)}$ and $F_{\tau(p_2)}$, then there exists a point on $s$ that is in $F_{\tau(p_1)} \cap F_{\tau(p_2)}$. Pulling-back using $\tau^{-1}$, this means the segment $\tau^{-1}(s)$ contains a point in $K_{p_1}\cap K_{p_2}$ which is a contradiction to condition (ii). Hence, we conclude that $s$ contains a point $q\not\in F_{\tau(p_1)}\cup F_{\tau(p_2)}$. And we can find $q^+$ very close to $q$ on the $\Omega$-side of $s$. \qedhere
    \end{itemize}
\end{proof}

We shall deduce the existence of a large convex polygon by combining a lower cup and a higher cap in many situations. The following facts on combination are quite useful (see \Cref{fig:combine}). Recall that when we denote by $\overline{p_1 \dotsb p_m}$ an $m$-cup or cap, we implicitly assume $x(p_1) < \dots < x(p_m)$.

\begin{figure}[!ht]
    \centering
\definecolor{ffqqqq}{rgb}{0.9,0.1,0.}
\definecolor{qqqqff}{rgb}{0.,0.1,0.9}
\definecolor{sqsqsq}{rgb}{0.12549019607843137,0.12549019607843137,0.12549019607843137}
\scalebox{0.75}{
\begin{tikzpicture}[line cap=round,line join=round,>=triangle 45,x=0.7cm,y=0.7cm]
\clip(2.38,1.58) rectangle (9.28,6.48);
\draw [line width=2.pt,color=qqqqff] (2.86,4.98)-- (4.86,5.98);
\draw [line width=2.pt,color=qqqqff] (4.86,5.98)-- (6.86,5.98);
\draw [line width=2.pt,color=qqqqff] (6.86,5.98)-- (8.86,3.98);
\draw [line width=2.pt,color=ffqqqq] (3.86,2.98)-- (6.,2.);
\draw [line width=2.pt,color=ffqqqq] (6.,2.)-- (8.,2.);
\draw [line width=2.pt,color=ffqqqq] (8.,2.)-- (8.86,3.98);
\draw [line width=2.pt,color=ffqqqq] (3.86,2.98)-- (2.86,4.98);
\begin{scriptsize}
\draw [fill=black] (2.86,4.98) circle (1.5pt);
\draw [fill=black] (4.86,5.98) circle (1.5pt);
\draw [fill=black] (6.86,5.98) circle (1.5pt);
\draw [fill=black] (8.86,3.98) circle (1.5pt);
\draw [fill=black] (3.86,2.98) circle (1.5pt);
\draw [fill=black] (6.,2.) circle (1.5pt);
\draw [fill=black] (8.,2.) circle (1.5pt);
\end{scriptsize}
\end{tikzpicture}
\begin{tikzpicture}[line cap=round,line join=round,>=triangle 45,x=0.7cm,y=0.7cm]
\clip(2.38,1.58) rectangle (9.28,6.48);
\draw [line width=1.pt,domain=2.38:9.28] plot(\x,{(--17.06--1.*\x)/4.});
\draw [line width=1.pt,domain=2.38:9.28] plot(\x,{(--11.04--1.*\x)/5.});
\draw [line width=1.pt,domain=2.38:9.28] plot(\x,{(--16.12-0.98*\x)/4.14});
\draw [line width=1.pt,domain=2.38:9.28] plot(\x,{(-6.16--1.98*\x)/2.86});
\draw [line width=1.pt,domain=2.38:9.28] plot(\x,{(--10.16-0.98*\x)/2.14});
\draw [line width=1.pt,domain=2.38:9.28] plot(\x,{(--4.-0.*\x)/2.});
\draw [line width=1.pt,domain=2.38:9.28] plot(\x,{(-14.12--1.98*\x)/0.86});
\draw [line width=1.pt,domain=2.38:9.28] plot(\x,{(--7.1--1.*\x)/2.});
\draw [line width=1.pt,domain=2.38:9.28] plot(\x,{(--11.96-0.*\x)/2.});
\draw [line width=3.6pt,color=qqqqff] (2.86,4.98)-- (4.86,5.98);
\draw [line width=3.6pt,color=qqqqff] (4.86,5.98)-- (6.86,5.98);
\draw [line width=3.6pt,color=ffqqqq] (3.86,2.98)-- (6.,2.);
\draw [line width=3.6pt,color=ffqqqq] (6.,2.)-- (8.,2.);
\draw [line width=3.6pt,color=ffqqqq] (8.,2.)-- (8.86,3.98);

\begin{scriptsize}
\draw [fill=black] (2.86,4.98) circle (1.5pt);
\draw [fill=black] (4.86,5.98) circle (1.5pt);
\draw [fill=black] (6.86,5.98) circle (1.5pt);
\draw [fill=black] (8.86,3.98) circle (1.5pt);
\draw [fill=black] (3.86,2.98) circle (1.5pt);
\draw [fill=black] (6.,2.) circle (1.5pt);
\draw [fill=black] (8.,2.) circle (1.5pt);
\end{scriptsize}
\end{tikzpicture}
\begin{tikzpicture}[line cap=round,line join=round,>=triangle 45,x=0.7cm,y=0.7cm]
\clip(2.38,1.58) rectangle (9.28,6.48);
\fill[line width=2.pt,color=sqsqsq,fill=sqsqsq,fill opacity=0.10000000149011612] (2.86,4.98) -- (4.86,5.98) -- (6.86,5.98) -- (8.,2.) -- (6.,2.) -- (3.86,2.98) -- cycle;
\draw [line width=2.pt,color=qqqqff] (2.86,4.98)-- (4.86,5.98);
\draw [line width=2.pt,color=qqqqff] (4.86,5.98)-- (6.86,5.98);
\draw [line width=2.pt,color=qqqqff] (6.86,5.98)-- (8.86,3.98);
\draw [line width=2.pt,color=ffqqqq] (3.86,2.98)-- (6.,2.);
\draw [line width=2.pt,color=ffqqqq] (6.,2.)-- (8.,2.);
\draw [line width=2.pt,color=ffqqqq] (8.,2.)-- (8.86,3.98);
\begin{scriptsize}
\draw [fill=black] (2.86,4.98) circle (1.5pt);
\draw [fill=black] (4.86,5.98) circle (1.5pt);
\draw [fill=black] (6.86,5.98) circle (1.5pt);
\draw [fill=black] (8.86,3.98) circle (1.5pt);
\draw [fill=black] (3.86,2.98) circle (1.5pt);
\draw [fill=black] (6.,2.) circle (1.5pt);
\draw [fill=black] (8.,2.) circle (1.5pt);
\end{scriptsize}
\end{tikzpicture}
}
    \caption{Three different ways of combining cups and caps.}
    \label{fig:combine}
\end{figure}
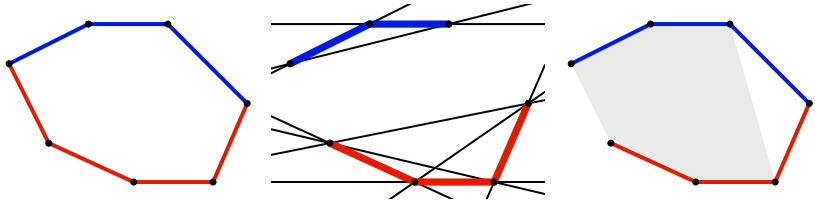

\begin{observation}\label{lem:cupcap_obv_combine}
    Let $C_- = \overline{sp_1 \dotsb p_kt}$ be a cup and $C_+ = \overline{sq_1 \dotsb q_\ell t}$ be a cap with $k, \ell \ge 0$. Then $s, p_1, \dots, p_k,t, q_l, \dots, q_{1}$ form a convex $(k+\ell+2)$-gon with vertices in this order.
\end{observation}

\begin{observation} \label{lem:cupcap_combine}
    Let $C_- = \overline{p_1 \dotsb p_k}$ be a cup and $C_+ = \overline{q_1 \dotsb q_\ell}$ be a cap with $k, \ell \ge 2$. If
    \begin{itemize}
        \item $\lin(p_i, p_j)$ is below $C_+$ for every $1 \le i < j \le k$, and 
        \item $\lin(q_i, q_j)$ is above $C_-$ for every $1 \le i < j \le \ell$, 
    \end{itemize}
    then $p_1, \dots, p_k, q_{\ell}, \dots, q_1$ form a convex $(k+\ell)$-gon with vertices in this order.
\end{observation}

\begin{proof}
    Pick $s,t\in \{p_1,\dots,p_k,q_1,\dots,q_\ell\}$ such that $s$ has the smallest $x$ coordinate and $t$ has the largest. We claim that $\{s\}\cup \{p_1,\dots, p_k\}$ forms a cup. Indeed, if $s \notin \{p_1, \dots, p_k\}$, then $s$ lies above any line $\lin(p_i, p_j)$ and has a smaller $x$-coordinate than $p_1, \dots, p_k$, so it extends the cup. Similarly, $\{t\} \cup \{p_1, \dots, p_k\}$ forms a cup, and hence $\{s,t\} \cup \{p_1,\dots, p_k\}$ is a cup as well. A similar argument shows $\{s,t\} \cup \{q_1,\dots, q_\ell\}$ forms a cap, and so the lemma follows from \Cref{lem:cupcap_obv_combine}. 
\end{proof}

\begin{observation} \label{lem:cupcap_combine_common}
    Suppose $p_1, \dots, p_k, q_\ell, \dots, q_{1}$ form a convex $(k+\ell)$-gon with vertices in this order, where $k \ge 2, \, \ell \ge 2$. If $t$ is a point such that $C_- = \overline{p_1 \dotsb p_kt}$ is a $(k+1)$-cup and $C_+ = \overline{q_1 \dotsb q_{\ell}t}$ is a $(\ell+1)$-cap, then $p_1, \dots, p_k, t, q_\ell, \dots, q_{1}$ form a convex $(k+\ell+1)$-gon with vertices in this order. 
\end{observation}

\begin{proof}
    Pick $s\in \{p_1,\dots,p_k,q_1,\dots,q_\ell\}$ such that $s$ has the smallest $x$-coordinate. We claim that $\{s\}\cup \{p_1,\dots, p_k, t\}$ forms a cup. From the definition of $C_-$ we see that $\overline{p_1 \dotsb p_k}$ is a cup. If $s$ is not in $p_1, \dots, p_k$, then, since $\{p_1,\dots,p_k,q_1,\dots,q_\ell\}$ is convex, $s$ lies above any line $\lin(p_i, p_j)$. Also, $s$ has a smaller $x$-coordinate than $p_1, \dots, p_k$, and so it extends the cup $\{p_1, \dots, p_k, t\}$. Similarly, $\{s\}\cup \{q_1,\dots, q_\ell, t\}$ forms a cup, and so the lemma follows from \Cref{lem:cupcap_obv_combine}. 
\end{proof}

We are ready to construct an $n$-gon-saturated set of small size. Let $p_i \eqdef (i, i^2)$ and $\veps \gg \delta>0$ be some sufficiently small constants depending on $n$ that will be fixed later. For each $i$ we will place a ``small and flat'' $(i+1, n+1-i)$-cup-cap-saturated set in an $\delta$-neighborhood of $p_i$. To be precise, for any fixed $\veps$ we can pick $\delta$ and $P_i$ such that 
\begin{enumerate}[label=\Roman*)]
    \item $P_i$ is a $(i+1, n+1-i; \pi/3)$-cup-cap-saturated set; 
    \item $P_i\subset \bB(p_i, \delta)$, where $\bB(x, r)$ denotes the open disk of radius $r$ around $x$; 
    \item $\lvert \slp(x, y) \rvert < \veps$ holds for any distinct $x, y \in P_i$; 
    \item $\lin(x, y)$ is below $P_j$ and $\lin(z, w)$ is above $P_i$ for any distinct $x, y \in P_i$ and any distinct $z, w \in P_j$ with $i < j$; 
    \item $\lin(x,y)$ and $\lin(z,w)$ intersect in $\bB(p_i, \veps)$ for any $(x, y, z, w) \in P_i \times P_{i+1} \times P_{i+1} \times P_{i+2}$. 
\end{enumerate}
 
For any $\veps$ and $\delta$ properties I), II), III), IV) can be achieved by sufficient scaling and flattening of an arbitrary very generic $(i+1, n+1-i)$-cup-cap-saturated set. In particular, I) follows from the very generic property and \Cref{lem:rotate}. Finally, by picking $\delta$ small enough with respect to $\veps$, V) follows from II). We will refer to the property IV) as the \emph{height hierarchy} of $P$.

We want one more property which needs some explanation. Suppose we have a cup (cap) $C$ of size at least $3$ in some $P_i$. Consider where a point $x$ in the plane might lie such that $x$ extends $C$ to a bigger cup (cap) but $x$ is not the first or the last point in the extended cup (cap). Since the size of the cup (cap) is at least $3$, the possible positions of $x$ are bounded regions in the plane. Therefore, by further flattening the sets $P_1, \dots, P_{n-1}$ we may assume that this bounded region is in $\bB(p_i, \delta)$. 
\begin{enumerate}[label=\Roman*)]
    \item[VI)] For any $P_i$ and cup (cap) $C$ with $|C|\ge 3$ within $P_i$, all the points of the plane that extend $C$ not at the ends lie in $\bB(p_i, \delta)$. 
\end{enumerate}
According to \Cref{thm:sat_csmall}, $P_1$ and $P_{n-1}$ each contains a single point. Hence, we pick $P_1 = \{p_1\}$ and $P_{n-1} = \{p_{n-1}\}$ for convenience. Note that the angle between any vertical line and any $\lin(p_i, p_{i+1})$ is less than $\pi/3$. Set $P \eqdef \bigcup_{i=1}^{n-1} P_i$. Our aim is to show that $P$ is $n$-gon-saturated. 
    
For this purpose, we will need three different arguments based on the position of the point $q$. The first case is when $q$ is close to one of the $P_2, \dots, P_{n-2}$. Denote by $\bD_i \eqdef \bB(p_i, \veps)$ for $i = 2, \dots, n-2$. We will refer to them as \emph{disks}, and we have $n-3$ disks in total. Recall that a point $q$ is generic (with respect to $P$) if $P \cup \{q\}$ is generic, and we call $q$ \textit{good} if $q$ is part of an $n$-gon in $P \cup \{q\}$. 

\begin{proposition} \label{prop:disk}
    For $i = 2, 3, \dots, n-2$, every generic point $q \in \bD_i$ is good. 
\end{proposition}

\begin{proof}
    Let $a\in P_{i-1}$ be the counterclockwise next vertex on the boundary of $\conv(P_{i-1} \cup \{q\})$ after $q$, and $b\in P_{i+1}$ be the clockwise next vertex on the boundary of $\conv(P_{i+1} \cup \{q\})$ after $q$. Then $\ell_-\eqdef \lin(a, q)$ and $\ell_+\eqdef \lin(b, q)$ partition the disk $\bD_i$ into four open regions, denote them as $\mathcal{R}_1, \mathcal{R}_2, \mathcal{R}_3$, and $\mathcal{R}_4$ according to  \Cref{fig:di_case}. If some point $p \in P_i$ is in  $\mathcal{R}_1 \cup \mathcal{R}_3$, then any sequence of points $q_i\in P_i$ for $i \in \{1,2,\dots,i-2,i+2,i+3,\dots,n-1\}$ together with $a, p, q, b$ form a convex $n$-gon. To see this note that if $\veps$ is small enough, so $q$ is very close to $p$ and therefore we only need that $a$ and $b$ lie on the same side of $\lin(p,q)$. This holds exactly when $p$ is in $\mathcal{R}_1 \cup \mathcal{R}_3$. 

    Assume that $P_i$ lies entirely in $\mathcal{R}_2 \cup \mathcal{R}_4$ then. In this case, we are going to find an $n$-gon from $P_{i-1} \cup P_i \cup P_{i+1} \cup \{q\}$ using the saturation properties of these sets. The idea is simple, $q$ is part of either a large cup or a large cap in $P_i$, and we try to combine that with either a cup from $P_{i-1}$ or a cap from $P_{i+1}$. To make the idea work, we start with a careful rotation of the point set $P$. 
    
	Let $\tau$ denote a rotation around $q$ with angle $\theta$. For any $S\subset \mathbb{R}^2$, denote by $S'$ the image of $S$ under $\tau$. We will show that $\theta$ can be chosen such that the following properties hold: 
    \begin{enumerate}[label=(\alph*)]
        \item $|\theta| < \frac{\pi}{3}$. 
        \item  $x(q) = x(\tau(q)) > x(\tau(p))$ holds for all $p \in P_{i-1} \cup P_{i+1}$.
        \item $P_{i-1}', P_i', P_{i+1}'$ obey the same height hierarchy as $P_{i-1}, P_i, P_{i+1}$.
        \item $P' \cup \{q\}$ is generic.
    \end{enumerate}
    Since $\veps$ is sufficiently small, our definition of $P_i$ implies that $\theta$ can be chosen to meet (a) and (b). (c) is ensured by III), (a), and $\veps$ being sufficiently small, as no line from those sets passes through a vertical state during the rotation. Finally, (d) is easy as it only forbids finitely many values for $\theta$. From (b) it follows that for any point $p \in \mathcal{R}_2 \cup \mathcal{R}_4$, in particular for any $p\in P_i'$, $\lin(q, p)$ is above $P_{i-1}'$ and below $P_{i+1}'$. This is the crucial advantage of applying the rotation $\tau$ (see the second part of \Cref{fig:di_case}).

    From (a) and I) we deduce that $P_i'$ is $(i+1, n+1-i)$-cup-cap-saturated, hence (d) implies that the point $q = q'$ is part of some $(i+1)$-cup or $(n+1-i)$-cap in $P_i' \cup \{q\}$.  Assume without loss of generality that $C_+$ is such a cap. Since $P_{i-1}'$ is $(i, n+2-i)$-cup-cap-saturated by I), from \Cref{prop:cc_satexist} it follows that there exists an $(i-1)$-cup $C_-$ in $P_{i-1}'$. We claim that $C_+$ and $C_-$ together form a convex $n$-gon. From \Cref{lem:cupcap_combine} we know that it is enough to check that all lines spanned by $C_+$ run above $C_-$ and all lines spanned by $C_-$ run below $C_+$. This follows from (c) for almost all the cases, except for the lines $\lin(q, p)$ with $p\in C_+$. We have seen that  $\lin(q, p)$ is above $P_{i-1}'$ for all $p\in P_{i}'$, hence these cases are also satisfied. We conclude that $C_+ \cup C_-$ forms a convex $n$-gon, and the proof is complete. 
    \end{proof}

    \begin{figure}[!ht]
        \centering
            \definecolor{qqqqff}{rgb}{0.,0.1,0.9}
    \definecolor{wqwqwq}{rgb}{0.3764705882352941,0.3764705882352941,0.3764705882352941}
    \definecolor{ffqqqq}{rgb}{0.9,0.1,0.}
    \scalebox{0.75}{
    \begin{tikzpicture}[line cap=round,line join=round,>=triangle 45,x=0.4cm,y=0.4cm]
\clip(-0.9895018390923995,-0.6817188017059387) rectangle (44.97616222827384,20.61786356190461);
\fill[line width=1.pt,fill=black,fill opacity=0.25] (11.637871192185218,5.9862507249539965) -- (13.952639636232503,6.480335382701773) -- (11.884404398257125,8.340288527726653) -- cycle;
\fill[line width=1.pt,fill=black,fill opacity=0.25] (11.637871192185218,5.9862507249539965) -- (9.32310274813793,5.492166067206221) -- (11.391337986113312,3.632212922181339) -- cycle;
\fill[line width=1.pt,fill=black,fill opacity=0.25] (34.40538072859401,9.097796602360908) -- (35.663660795491474,11.102541940758616) -- (32.88845871131448,10.914722630594149) -- cycle;
\fill[line width=1.pt,fill=black,fill opacity=0.25] (34.40538072859401,9.097796602360908) -- (33.14710066169654,7.093051263963204) -- (35.92230274587354,7.280870574127668) -- cycle;

\draw [line width=1.pt] (11.637871192185218,5.9862507249539965) circle (0.4*2.3669120386230738cm);

\draw [line width=1.5pt] (3.2419511658758164,2.8774550280885647)-- (3.3124743013146416,4.209204349529952);
\draw [line width=1.5pt] (3.3124743013146416,4.209204349529952)-- (4.59886321327134,3.2614086239863944);
\draw [line width=1.5pt] (4.59886321327134,3.2614086239863944)-- (5.272034210585329,1.7712421115103263);
\draw [line width=1.5pt] (5.272034210585329,1.7712421115103263)-- (4.105040364524771,1.6132116948562922);
\draw [line width=1.5pt] (4.105040364524771,1.6132116948562922)-- (3.2419511658758164,2.8774550280885647);
\draw [line width=1.5pt] (12.621966476088765,15.382946353710604)-- (14.295285817022604,16.138731147102902);
\draw [line width=1.5pt] (14.295285817022604,16.138731147102902)-- (15.942489364898947,16.158340713149045);
\draw [line width=1.5pt] (15.942489364898947,16.158340713149045)-- (16.393509383960325,15.177862410841698);
\draw [line width=1.5pt] (16.393509383960325,15.177862410841698)-- (15.491469345837567,15.09942414665711);
\draw [line width=1.5pt] (15.491469345837567,15.09942414665711)-- (12.621966476088765,15.382946353710604);
\draw (8.5,4.8) node[anchor=north west] {$\mathcal{R}_1$};
\draw (13.3,4.7) node[anchor=north west] {$\mathcal{R}_2$};
\draw (13.2,9.2) node[anchor=north west] {$\mathcal{R}_3$};
\draw (8.4,9.2) node[anchor=north west] {$\mathcal{R}_4$};
\draw [shift={(11.637871192185218,5.9862507249539965)},line width=1.pt,fill=black,fill opacity=0.25]  plot[domain=0.21029307036080222:1.4664489060205446,variable=\t]({1.*2.366912038623072*cos(\t r)+0.*2.366912038623072*sin(\t r)},{0.*2.366912038623072*cos(\t r)+1.*2.366912038623072*sin(\t r)});
\draw [shift={(11.637871192185218,5.9862507249539965)},line width=1.pt,fill=black,fill opacity=0.25]  plot[domain=3.351885723950595:4.608041559610338,variable=\t]({1.*2.3669120386230733*cos(\t r)+0.*2.3669120386230733*sin(\t r)},{0.*2.3669120386230733*cos(\t r)+1.*2.3669120386230733*sin(\t r)});
\draw (12.0,7.9) node[anchor=north west] {$p$};
\draw (15.195591142374587,17.963508312944025) node[anchor=north west] {$P_{i+1}$};
\draw (11.3,16.2) node[anchor=north west] {$b$};
\draw (2.9,1.7) node[anchor=north west] {$P_{i-1}$};
\draw (2.7,5.5) node[anchor=north west] {$a$};
\draw [line width=1.5pt] (11.637871192185218,5.9862507249539965)-- (13.952639636232503,6.480335382701773);

\draw [line width=1.pt] (34.40538072859401,9.097796602360908) circle (0.4*2.3669120386230738cm);

\draw [shift={(34.40538072859401,9.097796602360908)},line width=1.pt,fill=black,fill opacity=0.25]  plot[domain=0.21029307036080222:1.4664489060205446,variable=\t]({0.6967067093471654*2.366912038623072*cos(\t r)+-0.7173560908995228*2.366912038623072*sin(\t r)},{0.7173560908995228*2.366912038623072*cos(\t r)+0.6967067093471654*2.366912038623072*sin(\t r)});
\draw [shift={(34.40538072859401,9.097796602360908)},line width=1.pt,fill=black,fill opacity=0.25]  plot[domain=3.351885723950595:4.608041559610338,variable=\t]({0.6967067093471654*2.3669120386230733*cos(\t r)+-0.7173560908995228*2.3669120386230733*sin(\t r)},{0.7173560908995228*2.3669120386230733*cos(\t r)+0.6967067093471654*2.3669120386230733*sin(\t r)});
\draw [line width=1.5pt] (30.786000443630662,0.909013412786841)-- (29.879795897967803,1.887442300959961);
\draw [line width=1.5pt] (29.879795897967803,1.887442300959961)-- (31.455938720405037,2.149905581140942);
\draw [line width=1.5pt] (31.455938720405037,2.149905581140942)-- (32.993921494950825,1.5945998889944946);
\draw [line width=1.5pt] (32.993921494950825,1.5945998889944946)-- (32.29423313456775,0.6473488939169059);
\draw [line width=1.5pt] (32.29423313456775,0.6473488939169059)-- (30.786000443630662,0.909013412786841);
\draw [line width=1.5pt] (28.350229671909037,16.350484238542524)-- (28.973875658668817,18.077410396160364);
\draw [line width=1.5pt] (28.973875658668817,18.077410396160364)-- (30.107426380491624,19.272704010412497);
\draw [line width=1.5pt] (30.107426380491624,19.272704010412497)-- (31.12500713597657,18.913140156616947);
\draw [line width=1.5pt] (31.12500713597657,18.913140156616947)-- (30.552817955889072,18.2114077761074);
\draw [line width=1.5pt] (30.552817955889072,18.2114077761074)-- (28.350229671909037,16.350484238542524);
\draw [line width=1.5pt] (0.,20.)-- (20.,20.);
\draw [line width=1.5pt] (20.,0.)-- (0.,0.);
\draw [line width=1.5pt] (0.,0.)-- (0.,20.);
\draw [line width=1.5pt] (20.,0.)-- (20.,20.);
\draw [line width=1.5pt] (24.,20.)-- (44.,20.);
\draw [line width=1.5pt] (44.,20.)-- (44.,0.);
\draw [line width=1.5pt] (44.,0.)-- (24.,0.);
\draw [line width=1.5pt] (24.,0.)-- (24.,20.);
\draw [line width=1.5pt] (13.105500358461022,20.)-- (11.010944308583378,0.);
\draw [line width=1.5pt] (0.,3.5021605259146025)-- (20.,7.7711374857342825);
\draw [line width=1.5pt] (28.69514117937138,0.)-- (41.248157692712205,20.);
\draw [line width=1.5pt] (25.3033082645765,20.)-- (42.00098327949795,0.);
\draw (30.2,10.1) node[anchor=north west] {$\mathcal{R}_4'$};
\draw (36.8,10.1) node[anchor=north west] {$\mathcal{R}_2'$};

\draw [line width=1.5pt,color=ffqqqq] (30.382330198435916,1.6711134118554085)-- (31.005449093595857,1.1873104074623768);
\draw [line width=1.5pt,color=ffqqqq] (31.005449093595857,1.1873104074623768)-- (31.764051104922263,1.0436486947558024);
\draw [line width=1.5pt,color=ffqqqq] (32.51512258462475,1.4156752901214347)-- (31.764051104922263,1.0436486947558024);
\draw [line width=1.5pt,color=qqqqff] (33.065266161798334,9.019928754058189)-- (34.40538072859401,9.097796602360908);
\draw [line width=1.5pt,color=qqqqff] (34.40538072859401,9.097796602360908)-- (35.505554001727575,8.85585159954342);
\draw [line width=1.5pt,color=qqqqff] (35.505554001727575,8.85585159954342)-- (36.06048955744505,8.376266723663228);
\draw (11.7,6.0) node[anchor=north west] {$q$};
\draw (33.8,10.6) node[anchor=north west] {$q$};
\draw (33.9,6.6) node[anchor=north west] {$\mathcal{R}_1'$};
\draw (33.9,13.0) node[anchor=north west] {$\mathcal{R}_3'$};
\draw (31.0,18.805133147980307) node[anchor=north west] {$P_{i+1}'$};
\draw (33.25815490969175,2.49055942266159) node[anchor=north west] {$P_{i-1}'$};
\begin{scriptsize}
\draw [fill=black] (12.621966476088765,15.382946353710604) circle (1.5pt);
\draw [fill=black] (14.295285817022604,16.138731147102902) circle (1.5pt);
\draw [fill=black] (15.942489364898947,16.158340713149045) circle (1.5pt);
\draw [fill=black] (16.393509383960325,15.177862410841698) circle (1.5pt);
\draw [fill=black] (15.491469345837567,15.09942414665711) circle (1.5pt);
\draw [fill=black] (3.2419511658758164,2.8774550280885647) circle (1.5pt);
\draw [fill=black] (4.105040364524771,1.6132116948562922) circle (1.5pt);
\draw [fill=black] (5.272034210585329,1.7712421115103263) circle (1.5pt);
\draw [fill=black] (4.59886321327134,3.2614086239863944) circle (1.5pt);
\draw [fill=black] (3.3124743013146416,4.209204349529952) circle (1.5pt);
\draw [fill=black] (11.637871192185218,5.9862507249539965) circle (1.5pt);
\draw [fill=black] (11.98998531748832,6.471304997014045) circle (1.5pt);
\draw [fill=black] (3.5074084738911813,3.6979905197384597) circle (1.5pt);
\draw [fill=black] (3.5944805568732767,2.9139235857779573) circle (1.5pt);
\draw [fill=black] (4.810099146519692,1.990053457646681) circle (1.5pt);
\draw [fill=black] (4.019947063249522,2.2696457332653566) circle (1.5pt);
\draw [fill=black] (14.170228946790157,15.716083327111557) circle (1.5pt);
\draw [fill=black] (14.736321913653445,15.691822199960272) circle (1.5pt);
\draw [fill=black] (15.528852067262049,15.627125860890184) circle (1.5pt);
\draw [fill=black] (34.40538072859401,9.097796602360908) circle (1.5pt);
\draw [fill=black] (34.40538072859401,9.097796602360908) circle (1.5pt);
\draw [fill=black] (30.786000443630662,0.909013412786841) circle (1.5pt);
\draw [fill=black] (29.879795897967803,1.887442300959961) circle (1.5pt);
\draw [fill=black] (29.879795897967803,1.887442300959961) circle (1.5pt);
\draw [fill=black] (31.455938720405037,2.149905581140942) circle (1.5pt);
\draw [fill=black] (31.455938720405037,2.149905581140942) circle (1.5pt);
\draw [fill=black] (32.993921494950825,1.5945998889944946) circle (1.5pt);
\draw [fill=black] (32.993921494950825,1.5945998889944946) circle (1.5pt);
\draw [fill=black] (32.29423313456775,0.6473488939169059) circle (1.5pt);
\draw [fill=black] (32.29423313456775,0.6473488939169059) circle (1.5pt);
\draw [fill=black] (30.786000443630662,0.909013412786841) circle (1.5pt);
\draw [fill=black] (28.350229671909037,16.350484238542524) circle (1.5pt);
\draw [fill=black] (28.973875658668817,18.077410396160364) circle (1.5pt);
\draw [fill=black] (28.973875658668817,18.077410396160364) circle (1.5pt);
\draw [fill=black] (30.107426380491624,19.272704010412497) circle (1.5pt);
\draw [fill=black] (30.107426380491624,19.272704010412497) circle (1.5pt);
\draw [fill=black] (31.12500713597657,18.913140156616947) circle (1.5pt);
\draw [fill=black] (31.12500713597657,18.913140156616947) circle (1.5pt);
\draw [fill=black] (30.552817955889072,18.2114077761074) circle (1.5pt);
\draw [fill=black] (30.552817955889072,18.2114077761074) circle (1.5pt);
\draw [fill=black] (28.350229671909037,16.350484238542524) circle (1.5pt);
\draw [fill=black] (30.382330198435916,1.6711134118554085) circle (1.5pt);
\draw [fill=black] (31.005449093595857,1.1873104074623768) circle (1.5pt);
\draw [fill=black] (32.51512258462475,1.4156752901214347) circle (1.5pt);
\draw [fill=black] (31.764051104922263,1.0436486947558024) circle (1.5pt);
\draw [fill=black] (29.18993668610411,17.69323851671136) circle (1.5pt);
\draw [fill=black] (29.60174132156607,18.082425864443497) circle (1.5pt);
\draw [fill=black] (30.200312709835952,18.605877823855835) circle (1.5pt);
\draw [fill=black] (33.065266161798334,9.019928754058189) circle (1.5pt);
\draw [fill=black] (35.505554001727575,8.85585159954342) circle (1.5pt);
\draw [fill=black] (36.06048955744505,8.376266723663228) circle (1.5pt);
\end{scriptsize}
\end{tikzpicture}
    }
        \caption{$P_{i-1}$, $P_i$ and $P_{i+1}$ before and after the rotation.}
        \label{fig:di_case}
    \end{figure}
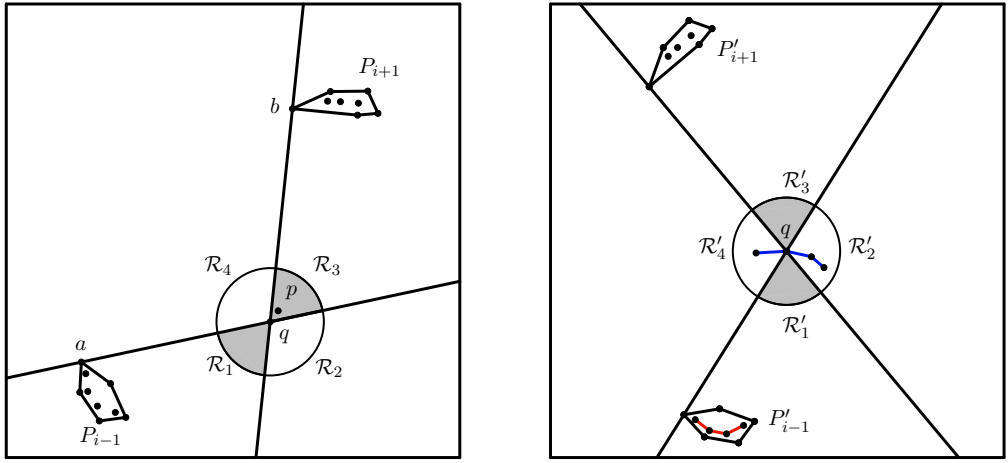

The next case is when the new point $q$ lies roughly between $P_i$ and $P_{i+1}$. Define 
\[
\bT_i \eqdef \conv(P_i \cup P_{i+1}) \setminus (\bD_i \cup \bD_{i+1}) 
\]
for $i = 1, \dots, n-2$. These regions are called \emph{tubes}. However, we need two more unbounded tube regions for technical reasons. From \Cref{thm:sat_csmall} we know that the set $P_2$ itself is an $(n-2)$-cap, and the set $P_{n-2}$ itself is an $(n-2)$-cup. Suppose $P_2$ is $\overline{p_{2}^1 \dotsb p_{2}^{n-2}}$, $P_{n-2}$ is $\overline{p_{n-2}^1 \dotsb p_{n-2}^{n-2}}$, and let 
\begin{itemize}
    \item $\bT_0$ denote the region enclosed by $\ray(p_2^1, p_1)$, $\ray(p_2^{n-2}, p_1)$  with apex $p_1$, 
    \item $\bT_{n-1}$ denote the region enclosed by $\ray(p_{n-2}^1, p_{n-1}), \ray(p_{n-2}^{n-2}, p_{n-1})$ with apex $p_{n-1}$. 
\end{itemize}
Note that we have defined $n$ \emph{tubes} $\bT_0, \dots, \bT_{n-1}$ in total.

The following observation follows from Carath\'eodory's theorem (Theorem~1.2.3~in~\cite{matousek2013lectures}).
    \begin{observation}\label{lem:tubetriangle}
    If $q \in \bT_i$ is a generic point where index $i \in \{1, \dots, n-2\}$, then there exists distinct $a, b, c \in P_i \cup P_{i+1}$, not all in $P_i$ or $P_{i+1}$, such that $q$ lies in the interior of $\conv(\{a, b, c\})$. 
\end{observation}

\begin{proposition} \label{prop:tube}
    For $i = 0, 1, \dots, n-1$, every generic point $q \in \bT_i$ is good. 
\end{proposition}
\begin{proof}
    We establish the existence of $q_1 \in P_1, \dots, q_{n-1} \in P_{n-1}$ such that $\overline{q_1 \dotsb q_i q q_{i+1} \dotsb q_{n-1}}$ gives an $n$-cup and hence a convex $n$-gon. Notice that since $\veps \ll 1$, for any selection $q_i\in P_i$ we have
    \begin{equation} \label{eq:mono_slope_global}
    \slp(q_1, q_2) < \dots < \slp(q_{n-2}, q_{n-1}). 
    \end{equation}
    That is, $\overline{q_1\dotsb q_{n-1}}$ is a cup. The boundary cases $i = 0$ or $n-1$ can be handled easily: 
    \begin{itemize}
        \item If $q \in \bT_0$, then \eqref{eq:mono_slope_global} implies that $q, p_1, p_2^1$ and any $q_k \in P_k \, (k = 3, \dots, n-1)$ form an $n$-cup. 
        \item If $q \in \bT_{n-1}$, then similarly, $q, p_{n-1}, p_{n-2}^{1}$ and any $q_k \in P_k \, (k = 1, \dots, n-3)$ form an $n$-cup. 
    \end{itemize}
    
    For $i\in \{1, \dots, n-2\}$, it is enough to pick $q_1, \dots, q_{n-1}$ such that 
    \[
    \slp(q_{i-1}, q_{i}) < \slp(q_i, q) < \slp(q, q_{i+1}) < \slp(q_{i+1}, q_{i+2}), 
    \]
    and so $\overline{q_1 \dotsb q_i q q_{i+1} \dotsb q_{n-1}}$ gives an $n$-cup. For the $i=1$ or the $i=n-2$ case, the inequality $\slp(q_{i-1}, q_{i}) < \slp(q_i, q)$ or the inequality $\slp(q, q_{i+1}) < \slp(q_{i+1}, q_{i+2})$ should be omitted, respectively. By \Cref{lem:tubetriangle}, we may assume without loss of generality that $q \in \conv(\{a, b, c\})$ with $a, b \in P_i$ and $c \in P_{i+1}$ satisfying $x(a) < x(b)$. See \Cref{fig:i_am_tired} and observe that
    \begin{equation} \label{eq:mono_slope_local}
        \slp(a, q) < \slp(a, c) < \slp(q, c) < \slp(b, c). 
    \end{equation}
    
    Given $i<n-2$ and $\varepsilon \ll 1$, we have $\slp(b, c) < \slp(c, d)$ for any $d \in P_{i+2}$. We also observe that any $e \in P_{i-1}$ satisfies $\slp(e, a) < \slp(a, q)$ if $i>1$. Indeed, let $x$ be the intersection of $\lin(e,a)$ and $\lin(c,b)$. From property V) of our construction, it follows that $x$ is in $\bD_i$, and so $\conv(\{a,b,x\}) \subseteq \bD_i$. Then $q \notin \bD_i$ implies that $\slp(e, a) < \slp(a, q)$. These inequalities together with \eqref{eq:mono_slope_global} and \eqref{eq:mono_slope_local} tells us that if we choose $q_i \eqdef a, \, q_{i+1} \eqdef c$ and all other $q_1, \dots, q_{i-1}, q_{i+2}, \dots, q_{n-1}$ arbitrarily, then  $\overline{q_1\dotsb q_i q q_{i+1}\dotsb q_{n-1}}$ is a cup.
\end{proof}
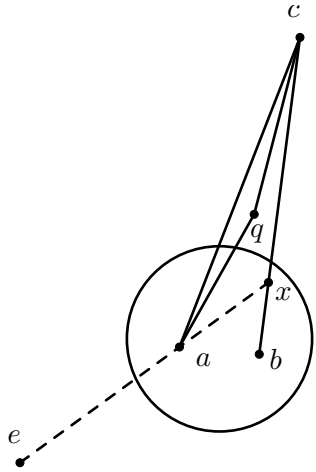
\begin{figure}[!ht]
        \centering
\scalebox{1.0}
{
       \begin{tikzpicture}[line cap=round,line join=round,>=triangle 45,x=0.3cm,y=0.3cm]
\clip(-6.000014267520551,-4.901672833292943) rectangle (11.170597306567084,16.19685642114154);
\draw [line width=1.pt] (2.8251865946115395,1.1404640343236605)-- (8.14048938172809,14.794028678160519);
\draw [line width=1.pt] (8.14048938172809,14.794028678160519)-- (6.3448698707123885,0.8218643580695969);
\draw [line width=1.pt] (8.14048938172809,14.794028678160519)-- (6.120417431835426,6.994306427186068);
\draw [line width=1.pt] (4.591565785670384,1.5016965582284665) circle (0.3*4.084143471790966cm);
\draw (-5.2,-2.1) node[anchor=north west] {$e$};
\draw (6.6,4.2) node[anchor=north west] {$x$};
\draw (3.1,1.3) node[anchor=north west] {$a$};
\draw (6.3448698707123885,1.5) node[anchor=north west] {$b$};
\draw [line width=1.pt] (2.8251865946115395,1.1404640343236605)-- (6.120417431835426,6.994306427186068);
\draw (5.5,7.0) node[anchor=north west] {$q$};
\draw (7.130453406781758,16.701874408614707) node[anchor=north west] {$c$};
\draw [line width=1.pt,dash pattern=on 4pt off 4pt] (-4.2076467020466675,-3.9596822495124497)-- (6.7517601307686315,3.9879791941322424);
\begin{scriptsize}
\draw [fill=black] (2.8251865946115395,1.1404640343236605) circle (1.5pt);
\draw [fill=black] (8.14048938172809,14.794028678160519) circle (1.5pt);
\draw [fill=black] (6.3448698707123885,0.8218643580695969) circle (1.5pt);
\draw [fill=black] (6.120417431835426,6.994306427186068) circle (1.5pt);
\draw [fill=black] (-4.2076467020466675,-3.9596822495124497) circle (1.5pt);
\draw [fill=black] (6.7517601307686315,3.9879791941322424) circle (1.5pt);
\end{scriptsize}
\end{tikzpicture}
}
        \caption{Points around $q$ when it lies in a tube.}
        \label{fig:i_am_tired}
\end{figure}

Our third case deals with any new point that is not in disks or tubes. Recall that $P_2$ is an $(n-2)$-cap $\overline{p_{2}^1 \dotsb p_{2}^{n-2}}$ and $P_{n-2}$ is an $(n-2)$-cup $\overline{p_{n-2}^1\dotsb p_{n-2}^{n-2}}$. For indices $i=2$ and $n-2$, we define $\overleftrightarrow{C_i}$ as the piecewise linear curve that consists of $\ray(p^2_i, p^1_i)$, the line segment $p^2_ip^{n-3}_i$, and $ \ray(p^{n-3}_i, p^{n-2}_i)$. Denote by $\bO$ the points of the plane that lie outside all disks $\bD \eqdef \bigcup_{i=2}^{n-2} \bD_i$ and all tubes $\bigcup_{i=0}^{n-1} \bT_i$. Then $\overleftrightarrow{C_2}$ and $\overleftrightarrow{C_{n-2}}$ partition $\bO$ into six connected regions $\oul$, $\our$, $\odl$, $\odr$, $\oll$, $\orr$, where the indices suggest their relative positions (such as upper-left, lower-right, and so on). \Cref{fig:blabla} illustrates the $n = 5$ case. 

\begin{figure}
    \centering
    \scalebox{0.8}{

\begin{tikzpicture}[line cap=round,line join=round,>=triangle 45,x=1.5cm,y=0.5cm]
\clip(-1.310499097057406,-2.731744975674936) rectangle (5.869601423777687,18.80855658683035);
\draw [line width=0.8pt,domain=-1.310499097057406:2.0] plot(\x,{(-0.3-0.05*\x)/-0.1});
\draw [line width=0.8pt,domain=2.0:5.869601423777687] plot(\x,{(--0.5-0.05*\x)/0.1});
\draw [line width=0.8pt,domain=-1.310499097057406:3.0] plot(\x,{(-1.05--0.05*\x)/-0.1});
\draw [line width=0.8pt,domain=3.0:5.869601423777687] plot(\x,{(--0.75--0.05*\x)/0.1});
\draw [line width=0.8pt,domain=2.9:5.869601423777687] plot(\x,{(-10.2--6.95*\x)/1.1});
\draw [line width=0.8pt,domain=3.1:5.869601423777687] plot(\x,{(-13.4--6.95*\x)/0.9});
\draw [line width=0.8pt,domain=-1.310499097057406:1.9] plot(\x,{(--2.05-2.95*\x)/-0.9});
\draw [line width=0.8pt,domain=-1.310499097057406:2.1] plot(\x,{(--1.85-2.95*\x)/-1.1});
\draw [rotate around={90.:(2.,4.)},line width=0.8pt] (2.,4.) ellipse (0.375cm and 0.375cm);
\draw [rotate around={90.:(3.,9.)},line width=0.8pt] (3.,9.) ellipse (0.375cm and 0.375cm);
\draw [line width=0.8pt] (2.9,9.05)-- (1.9,3.95);
\draw [line width=0.8pt] (3.1,9.05)-- (2.1,3.95);
\draw [fill=black] (1.,1.) circle (1.0pt);
\draw[color=black] (1.1998019425625505,0.7228316900098742) node {$P_1$};
\draw [fill=black] (2.,4.) circle (1.0pt);
\draw[color=black] (2.441643227612645,4.583829139892897) node {$P_2$};
\draw [fill=black] (1.9,3.95) circle (1.0pt);
\draw [fill=black] (2.1,3.95) circle (1.0pt);
\draw [fill=black] (3.,9.) circle (1.0pt);
\draw[color=black] (3.412537323197265,10.002772929202404) node {$P_3$};
\draw [fill=black] (2.9,9.05) circle (1.0pt);
\draw [fill=black] (3.1,9.05) circle (1.0pt);
\draw [fill=black] (4.,16.) circle (1.0pt);
\draw[color=black] (4.202799959138234,15.794269104026938) node {$P_4$};
\draw[color=black] (-0.5,16) node {$\oul$};
\draw[color=black] (5,16) node {$\our$};
\draw[color=black] (5,7) node {$\orr$};
\draw[color=black] (-0.5,7) node {$\oll$};
\draw[color=black] (5,0.25) node {$\odr$};
\draw[color=black] (-0.5,0.25) node {$\odl$};
\end{tikzpicture}
}
    \caption{Regions around the construction.}
    \label{fig:blabla}
\end{figure}
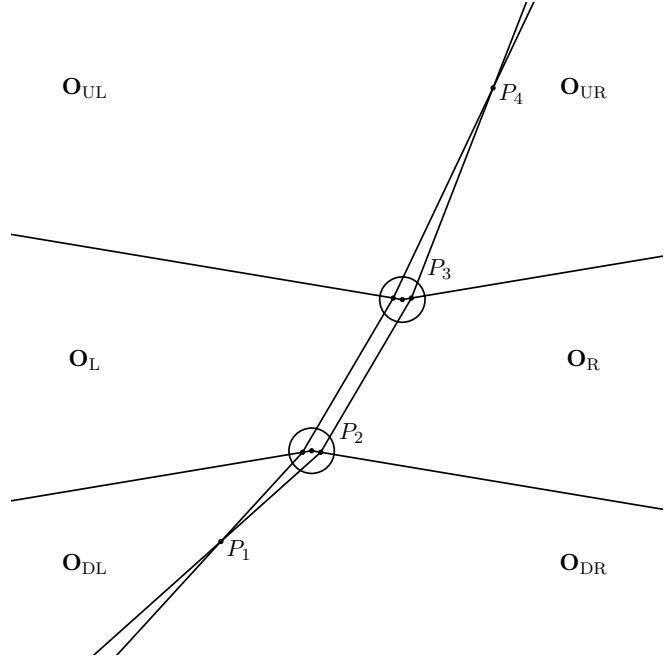

\begin{proposition} \label{prop:outer}
    Every generic point of $\bO$ is good. 
\end{proposition}
\begin{proof}
    Observe that every generic point from $\odl \cup \odr$ is part of an $n$-gon with $P_1 \cup P_2$. Also, every generic point from $\oul \cup \our$ is part of an $n$-gon with $P_{n-2} \cup P_{n-1}$. Hence, we can assume $q \in \oll \cup \orr$. We suppose further that $q \in \orr$, as the proof of the other case is similar. 

    For $i = 2, \dots, n-2$, the saturation properties imply that $q$ is part of either an $(i+1)$-cup or an $(n+1-i)$-cap in $P_i \cup \{q\}$. Since $q$ is part of a $3$-cup in $P_2 \cup \{q\}$ and a $3$-cap in $P_{n-2} \cup \{q\}$, it follows that there exists some $j \in \{2, \dots, n-3\}$ together with
    \begin{itemize}
        \item a $(j+1)$-cup $C_- \eqdef \overline{a_1 \dotsb a_{j+1}}$ from $P_j \cup \{q\}$ containing $q$, and
        \item an $(n-j)$-cap $C_+ \eqdef \overline{b_1 \dotsb b_{n-j}}$ from $P_{j+1} \cup \{q\}$ containing $q$. 
    \end{itemize}

    Using property VI) of our construction and the fact that $q\notin D_i$ for any $i$, we conclude that $q$ must be either the first or last point of $C_-$ and also of $C_+$. Notice that here property VI) is only applicable to cups and caps of size at least 4, but we excluded the regions $\odl, \odr, \oul, \our$ at the beginning to avoid small cups and caps. 
    
    If $q$ is the first point of both or the last point of both, then from \Cref{lem:cupcap_combine_common} it follows that $C_-\cup C_+\cup \{q\}$ is a convex $n$-gon. Thanks to $q \in \orr$, we may assume $b_1 = q$ then. We claim that \Cref{lem:cupcap_combine} is applicable to $C_-$ and $C_+' \eqdef \overline{b_2 \dotsb b_{n-j}}$ (so $a_1, \dots, a_j, q, b_{n-j}, \dots, b_2$ form an $n$-gon). Due to the height hierarchy of $P$, it suffices to show that the point $q$ is below $\lin(b_s, b_t)$ for any $1 \le s < t \le n-j-1$, and $\lin(a_k, q)$ is below $C_+'$ for any $1 \le k \le j$. Indeed, the former is a result of $C_+$ being a cap, and the latter follows from the facts $q \in \orr$ and
    \[
    \min\{x(p) : p \in P_j\} < x(q) < \max\{x(p) : p \in P_{j+1}\}. \qedhere
    \]
\end{proof}

Finally, we are ready to finish the proof and hence conclude this section.
\begin{proof}[Proof of \Cref{thm:gconstruction}]
    Let $P$ be as described in this section. By its construction, we clearly have $|P|=\sum_{i=1}^{n-1}|P_i|$. It then suffices to argue that $P$ is $n$-gon-saturated. 

    First, we show that $P$ is $n$-gon-free. This is parallel to the original Erd\H{o}s--Szekeres proof of their construction. Suppose that $G\subset P$ is in convex position. $G$ cannot be contained in a single $P_i$ because any $n$-gon contains either a $(i+1)$-cup or an $(n+1-i)$-cap. So, we can assume that $G$ intersects at least two of $P_1, \dots, P_{n-1}$. It follows from the height hierarchy that any four points $q_1 \in P_i, \, q_2, q_3 \in P_j, \, q_4 \in P_k$ with $i<j<k$ are not in convex position. Thus, there are at most two of $P_1, \dots, P_{n-1}$ containing more than one point of $G$, and the other points of $G$ are between these two groups. So, there exist $1 \le i_1 < i_2 \le n-1$ such that $|P_j \cap G| \le 1$ for $j = i_1+1, \dots, i_2-1$ and $|P_j\cap G|=0$ for $j<i_1$ and $j>i_2$. Since $G$ is not contained in a single group, $P_{i_1} \cap G$ must be a cup and $P_{i_2} \cap G$ must be a cap. From the saturation properties of $P_{i_1}, P_{i_2}$ we obtain
    \[
    |G|\le i_1+(i_2-i_1-1)\cdot 1+(n-i_2)=n-1. 
    \]

    We then argue that any $q \notin P$ is part of an $n$-gon as long as $P \cup \{q\}$ is in general position. Consider all lines spanned by $P$. These lines cut the plane into polygonal cells and since $P\cup \{q\}$ is in general position $q$ lies in the interior of a cell. If any other point in its cell is part of a convex $n$-gon, then $q$ is part of a convex $n$-gon with the same $n-1$ other vertices from $P$. Hence, if needed, we can move $q$ within its cell such that $P \cup \{q\}$ becomes generic. Then, by putting \Cref{prop:disk,prop:tube,prop:outer} together, the proof is complete.
\end{proof}

\section{Final remarks} \label{sec:remarks}

Our current upper bounds on $\sat_{\sf c}(k, \ell)$ and $\sat_{\sf g}(n)$ are exponential, while the lower bounds are linear ($\sat_{\sf g}(n) \ge n-1$ is trivial). The obvious problem is to obtain better bounds. 

\begin{problem}
    What is the correct asymptotics for $\sat_{\sf c}(k, \ell)$ and $\sat_{\sf g}(n)$? 
\end{problem}

It is known that the Ramsey number equals to $\binom{k+\ell-4}{k-2}$ for monotone paths in 3-uniform complete ordered hypergraphs with transitive 2-colorings, see e.g. \cite{MoshkovitzShapira2014}.

\begin{problem}
    Is the saturation number equal to $\binom{k+\ell-4}{k-2}$ again for monotone paths in $3$-uniform complete ordered hypergraphs with transitive $2$-colorings? 
\end{problem}

Inside a generic planar point set $P$, a subset $S$ is called an \textit{$n$-hole} if $S$ forms an $n$-gon whose convex hull contains no points of $P$ in its interior. A construction due to Horton~\cite{horton1983sets} shows that there are arbitrarily large point sets without $n$-holes for every $n \geq 7$. 

\begin{problem}
    For $n \ge 7$, is the saturation number for $n$-holes bounded (i.e., finite)? 
\end{problem}



\small
\bibliographystyle{alphaabbrv-url}
\bibliography{references}

\appendix

\section{Supplementary code for the proof of \texorpdfstring{\Cref{thm:sat45}}{Theorem~10}} \label{sec:code}

We used the following SageMath program to check that the set \begin{equation*}
    P=\{(-60, 40), (-40, 20), (-20, 16), (0, 10), (5, -50), (15, -40), (25, -40), (125, -230)\}.
\end{equation*} is saturated for $4$-cups and $5$-caps. Note that our conclusion is under the fact that $P$ contains no $4$-cups or $5$-caps which is not checked by this program itself.
\begin{lstlisting}[breaklines=true, frame=tb, language=Python, basicstyle={\footnotesize}]
# Use the HyperplaneArrangements package of SageMath
H.<x,y> = HyperplaneArrangements(QQ)

# Define the point set as vectors over the rational numbers
points=[vector(QQ,[-60, 40]), vector(QQ,[-40, 20]), vector(QQ,[-20, 16]), vector(QQ,[0, 10]), vector(QQ,[5, -50]), vector(QQ,[15, -40]), vector(QQ,[25, -40]), vector(QQ,[125, -230])]

# Generate all vertical lines and lines between points
lines=set()
for i in range(8):
    lines.add(x-points[i][0])
for i in range(8):
    for j in range(i+1,8):
        lines.add((points[j][0]-points[i][0])*(y - points[i][1])-(x-points[i][0])*(points[j][1]-points[i][1]))

# Generate and check all faces in this line arrangement
arrangement=H(lines)
faces=arrangement.regions()
flag_saturated = True
for f in faces:
    
    # Form a new point set with a representative point
    newpoints=copy(points)
    newpoints.append(f.representative_point())
    newpoints.sort()
    
    # Check the existence of 4-cups
    flag_4cup = False
    for i in range(9):
        for j in range(i+1,9):
            for k in range(j+1,9):
                for l in range(k+1,9):
                    slope1= (newpoints[j][1]-newpoints[i][1])/(newpoints[j][0]-newpoints[i][0])
                    slope2= (newpoints[k][1]-newpoints[j][1])/(newpoints[k][0]-newpoints[j][0])
                    slope3= (newpoints[l][1]-newpoints[k][1])/(newpoints[l][0]-newpoints[k][0])
                    if slope1 < slope2 and slope2 < slope3:
                        flag_4cup = True
                        
    # Check the existence of 5-caps
    flag_5cap = False
    for i in range(9):
        for j in range(i+1,9):
            for k in range(j+1,9):
                for l in range(k+1,9):
                    for m in range(l+1,9):
                        slope1= (newpoints[j][1]-newpoints[i][1])/(newpoints[j][0]-newpoints[i][0])
                        slope2= (newpoints[k][1]-newpoints[j][1])/(newpoints[k][0]-newpoints[j][0])
                        slope3= (newpoints[l][1]-newpoints[k][1])/(newpoints[l][0]-newpoints[k][0])
                        slope4= (newpoints[m][1]-newpoints[l][1])/(newpoints[m][0]-newpoints[l][0])
                        if slope1 > slope2 and slope2 > slope3 and slope3 > slope4:
                            flag_5cap = True

    # If no 4-cup or 5-cap in the new set, then the original set is not saturated.                   
    if flag_4cup == False and flag_5cap == False:
        flag_saturated = False

# Output the result
if flag_saturated == False:
    print("This set is not saturated.")
else:
    print("This set is saturated.")
\end{lstlisting}

\end{document}